\documentclass[12pt]{amsart}%
\usepackage{amsmath}
\usepackage{amsfonts}
\usepackage{amssymb}
\usepackage{amstext}
\usepackage{graphicx}%
\providecommand{\U}[1]{\protect \rule{.1in}{.1in}}
\providecommand{\U}[1]{\protect \rule{.1in}{.1in}}
\providecommand{\U}[1]{\protect \rule{.1in}{.1in}}
\newtheorem{theorem}{Theorem}[section]
\newtheorem{lemma}{Lemma}[section]

\newtheorem{corollary}{Corollary}[section]

\numberwithin{equation}{section}

\theoremstyle{remark}
\newtheorem{remark}{Remark}[section]
\numberwithin{equation}{section}
\begin{document}
\title[On the $p$-pseudoharmonic map heat flow]{On the $p$-pseudoharmonic map heat flow }
\author{$^{\ast}$Shu-Cheng Chang$^{1}$}
\address{$^{1}$Department of Mathematics and Taida Institute for Mathematical Sciences
(TIMS), National Taiwan University, Taipei 10617, Taiwan, R.O.C.}
\email{scchang@math.ntu.edu.tw }
\author{$^{\sharp}$Yuxin Dong$^{2}$}
\address{$^{2}$ School of Mathematics, Fudan University, Shanghai, 200433, P.R. China}
\email{yxdong@fudan.edu.cn}
\author{$^{\dag}$Yingbo Han$^{3}$}
\address{$^{3}${College of Mathematics and Information Science, Xinyang Normal
University}\\
Xinyang,464000, Henan, P.R. China}
\email{{yingbohan@163.com}}
\thanks{$^{\ast}$Research supported in part by the NSC of Taiwan}
\thanks{$^{\sharp}$Research supported in part by the NSFC No. 11271071}
\thanks{$^{\dag}$Research supported in part by the NSFC No. 11201400, Nanhu Scholars
P. for Young Scholars of XYNU}
\subjclass{Primary 32V05, 32V20; Secondary 53C56}
\keywords{$p$-pseudoharmonic map; $p$-pseudoharmonic map heat flow; Morse-type harnack
inequality; pseudohermitian manifold; Sasakian manifold; $p$-sublaplacian.}
\maketitle

\begin{abstract}
In this paper, we consider the heat flow for $p$-pseudoharmonic maps from a
closed Sasakian manifold $(M^{2n+1},J,\theta)$ into a compact Riemannian
manifold $(N^{m},g_{ij})$. We prove global existence and asymptotic
convergence of the solution for the $p$-pseudoharmonic map heat flow, provided
that the sectional curvature of the target manifold $N$ is nonpositive.
Moreover, without the curvature assumption on the target manifold, \ we obtain
global existence and asymptotic convergence of the $p$-pseudoharmonic map heat
flow as well when its initial $p$-energy is sufficiently small.

\end{abstract}

\section{ Introduction}

In the seminal paper of J. Eells and J. H. Sampson (\cite{es}), they proved
the existence theorem of harmonic maps between compact Riemannian manifolds
via the harmonic map heat flow when the target manifold with nonpositive
sectional curvature. In our previous papers (\cite{cc1}, \cite{cc2}), we
considered the following pseudoharmonic map heat flow from a closed
pseudohermitian manifold $(M^{2n+1},J,\theta)$ into a compact Riemannian
manifold $(N^{m},g_{ij})$ on $M\times \lbrack0,T)$ :
\begin{equation}
\left \{
\begin{split}
&  \frac{\partial \varphi^{k}}{\partial t}=\Delta_{b}\varphi^{k}+2h^{\alpha
\bar{\beta}}\widetilde{\Gamma}_{ij}^{k}\varphi_{\alpha}^{i}\varphi_{\bar
{\beta}}^{j},\quad k=1,\cdots,m,\\
&  \varphi(\omega,0)=u_{0}(\omega),\enspace u_{0}\in C^{\infty}(M;N),
\end{split}
\right.  \label{1}%
\end{equation}
for $\varphi \in C^{\infty}(M\times \lbrack0,T);N).$ Here $\Delta_{b}$ is the
sub-Laplace operator and $\widetilde{\Gamma}_{ij}^{k}$ are the Christoffel
symbols of $N$. Then we proved the \ pseudoharmonic map heat flow (\ref{1})
admits a unique, smooth solution $\varphi \in C^{\infty}(M\times \lbrack
0,\infty);N)$ with subconverges to a pseudoharmonic maps $\varphi_{\infty}\in
C^{\infty}(M,N)$ as $t\rightarrow \infty$, provided that $M$ is Sasakian (i.e.
vanishing pseudohermitian torsion) and the sectional curvature $K^{N}$ is
nonpositive. This served as the CR analogue to Eells-Sampson's Theorem
(\cite{es}) for the harmonic map heat flow. Secondly, without the curvature
assumption on the target manifold (\cite{cs}, \cite{cd}), we showed that there
exists $\varepsilon>0$ depending on $n,M,N$ and $||\nabla_{b}u_{0}%
||_{L^{\infty}(M)}$ such that for any initial data $u_{0}\in C^{\infty}(M;N)$,
if the energy is small enough
\[
E(u_{0})=\int_{M}|\nabla_{b}u_{0}|^{2}d\mu \leq \varepsilon,
\]
then the solution $\varphi$ of (\ref{1}) exists for all $t>0$. Moreover, as
$t\rightarrow \infty$, $\varphi(t)$ converges to a constant map. Here
$\nabla_{b}$ is the subgradient on the holomorphic subbundle $T_{1,0}M\oplus
T_{0,1}M.$

In this paper, we extend the above results to the $p$-pseudoharmonic map heat
flow (\ref{4}) on $M\times \lbrack0,T).$ Let $(M^{2n+1},J,\theta)$ be a closed
pseudohermitian manifold and $(N^{m},g_{ij})$ be a compact Riemannian
manifold. At each point $x\in M$, we may take a local coordinate chart
$U_{x}\subset M$ of $x$ and a local coordinate chart $V_{\varphi(x)}\subset N$
of $\varphi(x)$ such that $\varphi(U_{x})\subset V_{\varphi(x)}$. For a
$C^{1}$-map $\varphi:M\rightarrow N,$ we define the energy density
$e(\varphi)$ of $\varphi$ at the point $\omega \in U_{x}$ by
\[
e(\varphi)(\omega)=\frac{1}{2}h^{\alpha \overline{\beta}}(\omega)g_{ij}%
(\varphi(\omega))\varphi_{\alpha}^{i}\varphi_{\overline{\beta}}^{j}.
\]
Here $h_{\alpha \overline{\beta}}$ is the Levi metric on $(M^{2n+1},J,\theta).$
It can be checked that the energy density is intrinsically defined, i.e.,
independent of the choice of local coordinates. Its $p$-energy $E_{p}%
(\varphi)$ of $\varphi$ is defined by
\begin{equation}
E_{p}(\varphi)=\frac{1}{p}\int_{M}e(\varphi)^{\frac{p}{2}}d\mu=\frac{1}{p}%
\int_{M}|\nabla_{b}\varphi|^{p}d\mu,\  \  \ p>1 \label{2}%
\end{equation}
where $d\mu=\theta \wedge(d\theta)^{n}$ is the volume element on $M$. \ The
$p$-pseudoharmonic map is the critical point of (\ref{2}) which is the
solution of the Euler-Lagrange equation associated to its $p$-energy
$E_{p}(\varphi)$%
\begin{equation}
\triangle_{b,p}\varphi^{k}+2|\nabla_{b}\varphi|^{p-2}\widetilde{\Gamma}%
_{ij}^{k}\varphi_{\alpha}^{i}\varphi_{\bar{\alpha}}^{j}=0,\  \  \ k=1,...,m,
\label{3}%
\end{equation}
where $\Gamma_{ij}^{k}$ are the Christoffel symbols of $(N^{m},g_{ij})$ and
$\triangle_{b,p}$ is the $p$-sublaplacian
\[
\triangle_{b,p}\varphi^{k}=\operatorname{div}_{b}(|\nabla_{b}\varphi
|^{p-2}\nabla_{b}\varphi^{k}).
\]
For $p=2$, $\triangle_{b,2}$ is the usual sublaplacian. It is singular for
$p\neq2$ at points where $\nabla_{b}\varphi=0.$ Let $S^{1,p}(M,N)$ be the
Folland-Stein space (see next section for definition). We call a map
$\varphi \in S^{1,p}(M,N)$ is a weakly pseudoharmonic map if it is a weak
solution of (\ref{3}). In general it is far from understood about the
regularity of the weak $p$-pseudoharmonic map (\cite{f}, \cite[Theorem
21.1.]{fs}, \cite{jl}, \cite{hs}, \cite{dt}, \cite{xz}).

In this paper, we consider the associated $p$-pseudoharmonic map heat flow on
$M\times \lbrack0,T)\ $:
\begin{equation}
\left \{
\begin{array}
[c]{l}%
\frac{\partial \varphi^{k}}{\partial t}=\operatorname{div}_{b}(|\nabla
_{b}\varphi|^{p-2}\nabla_{b}\varphi^{k})+2|\nabla_{b}\varphi|^{p-2}%
h^{\alpha \bar{\beta}}\tilde{\Gamma}_{ij}^{k}\varphi_{\alpha}^{i}\varphi
_{\bar{\beta}}^{j},\  \  \ k=1,\cdots,m\\
\varphi(\omega,0)=u_{0}(\omega),
\end{array}
\right.  \label{4}%
\end{equation}
where $u_{0}:M\rightarrow N$ is the initial data which to be of class
$C^{2,\alpha}\ $for $0<\alpha<1.$ We will follow methods of \cite{cc1},
\cite{fr1} and \cite{fr2} to study the global weak solutions to the
$p$-pseudoharmonic map heat flow (\ref{4}) from a closed Sasakian
$(M^{2n+1},J,\theta)$ to a compact Riemannian manifold $(N,g_{ij}).$ In fact,
we first consider the following regularized problem of (\ref{4}) for
$0<\delta<1$,
\begin{equation}
\left \{
\begin{array}
[c]{l}%
\frac{\partial \varphi^{k}}{\partial t}=\operatorname{div}_{b}([|\nabla
_{b}\varphi|^{2}+\delta]^{\frac{p-2}{2}}\nabla_{b}\varphi^{k})+2[|\nabla
_{b}\varphi|^{2}+\delta]^{\frac{p-2}{2}}h^{\alpha \bar{\beta}}\tilde{\Gamma
}_{ij}^{k}\varphi_{\alpha}^{i}\varphi_{\bar{\beta}}^{j},\  \  \ k=1,\cdots,m\\
\varphi(\omega,0)=u_{0}(\omega),
\end{array}
\right.  \label{4a}%
\end{equation}
on $M\times \lbrack0,T_{\delta})$ for the regularized $p$-energy $E_{p,\delta
}:$
\[
E_{p,\delta}(\varphi)=\frac{1}{p}\int_{M}(|\nabla_{b}\varphi|^{2}%
+\delta)^{\frac{p}{2}}d\mu
\]
with the regularized energy density $e_{\delta}(\varphi):=|\nabla_{b}%
\varphi|^{2}+\delta.$

The main difficulty comes from the CR Bochner formula (\ref{CR}) with a mixed
term $\langle J\nabla_{b}\varphi,\nabla_{b}\varphi_{0}\rangle_{L_{\theta}}$
involving the covariant derivative of $\varphi$ in the direction of the
characteristic vector field $\mathbf{T}$, which has no analogue in the
Riemannian case. However, by adding an $\mathbf{T}$-energy density
$e_{0}(\varphi)$, we are able to overcome such a difficulty and conclude that
the $p$-pseudoharmonic map heat flow has a global smooth solution from a
closed Sasakian manifold $(M^{2n+1},J,\theta)$ into a compact Riemannian
manifold $(N^{m},g_{ij})$. More precisely, with the same spirit as in
\cite{cc1}, instead of the original energy density $e_{\delta}(\varphi),$ we
estimate the total energy density
\[
\widehat{e_{\delta}}(\varphi)=e_{\delta}(\varphi)+\varepsilon e_{0}(\varphi)
\]
by adding an $\mathbf{T}$-energy density
\[
e_{0}(\varphi)=g_{ij}\varphi_{0}^{i}\varphi_{0}^{j}%
\]
for some positive constant $\varepsilon$ which to be determined later. We
first are able to derive the Moser type Harnack inequality (Lemma \ref{l32}
and \cite[Theorem 1.1.]{cc1}) for the total (regularized) energy density
$\widehat{e}(\varphi)$ if $M$ is Sasakian. Secondly, based on \cite{cc1},
\cite{fr1} and \cite{fr2}, we show the energy density of the regularized
$p$-pseudoharmonic map heat flow (\ref{4a}) is uniformly bounded as following :

\begin{theorem}
\label{T31} Let $(M^{2n+1},J,\theta)$ be a closed Sasakian manifold and
$(N,g)$ be a compact Riemannian manifold. Let $u_{0}\in C^{2,\alpha}(M,N)$,
$0<\alpha<1$ and $||\nabla_{b}u_{0}||_{L^{\infty}(M)}\leq K$.

$(i)$\ There exists $\varepsilon_{0}>0$ depending on $K,\ M,N$ such that if
\begin{equation}
E_{p}(u_{0})=\frac{1}{p}\int_{M}|\nabla_{b}u_{0}|^{p}d\mu \leq \varepsilon_{0},
\label{t31a}%
\end{equation}
then the solution $\varphi_{\delta}$ of (\ref{4a}) satisfies
\begin{equation}
||\nabla_{b}\varphi_{\delta}||_{L^{\infty}([0,T^{\prime})\times M)}\leq
C\  \  \  \mathrm{and}\  \  \ ||\mathbf{T}\varphi_{\delta}||_{L^{\infty
}([0,T^{\prime})\times M)}\leq C,\  \label{t31}%
\end{equation}
where $C$ is a constant depending on $K,M$ and $N$.

$(ii)$ In addition, if \ the sectional curvature of $(N,g_{ij})$ is
nonpositive
\[
K^{N}\leq0,
\]
then (\ref{t31}) holds without the smallness assumption (\ref{t31a}).

$(iii)\ $The energy inequality will be
\begin{equation}
\int_{0}^{t}\int_{M}|\partial_{s}\varphi_{\delta}|^{2}(x,s)d\mu ds+E_{p,\delta
}(\varphi_{\delta}(\cdot,t))=E_{p,\delta}(u_{0})\leq E_{p,1}(u_{0}%
),\quad \forall \,t\in \lbrack0,T_{\delta}). \label{t31b}%
\end{equation}

\end{theorem}

Based on (\ref{t31}), (\ref{t31b}) and the CR divergence theorem and CR
Green's identity as in \cite[Lemma 3.2. and Corollary 3.1.]{ccw}, it follows
from \cite{df}, \cite{d}, \cite{ch}, \cite{hs} and \cite{lsu} that we can
prove the global existence and asymptotic convergence of the $p$%
-pseudoharmonic map heat flow, provided that the sectional curvature of the
target manifold $N$ is nonpositive.

\begin{theorem}
\label{t1} Let $(M^{2n+1},J,\theta)$ be a closed Sasakian manifold and $(N,g)$
be a compact Riemannian manifold. \ If the sectional curvature is nonpositive
\[
K^{N}\leq0
\]
and $u_{0}\in C^{2,\alpha}(M,N)$, $0<\alpha<1,$ then there is a unique global
weak solution $\varphi$ of (\ref{4}) with $\partial_{t}\varphi \in
L^{2}(M\times \lbrack0,\infty))$ and $\varphi,\nabla_{b}\varphi \in C^{\beta
}(M\times \lbrack0,\infty),N)$, where $0<\beta<1$. Moreover, there exists a
sequence $t_{k}\rightarrow \infty$ such that $\varphi(t_{k})$ converges in
$C^{1,\beta^{\prime}}(M,N)$ for all $\beta^{\prime}<\beta$, to a weakly
$p$-pseudoharmonic map $\varphi_{\infty}\in$ $C^{1,\beta}(M,N)$ satisfying
\[
E_{p}(\varphi_{\infty})\leq E_{p}(u_{0}).
\]

\end{theorem}

Moreover, without the curvature assumption on the target manifold but with
small initial $p$-energy, we have

\begin{theorem}
\label{t2} Let $(M^{2n+1},J,\theta)$ be a closed Sasakian manifold and $(N,g)$
be a compact Riemannian manifold. Let $u_{0}\in C^{2,\alpha}(M,N)$,
$0<\alpha<1$ and $||\nabla_{b}u_{0}||_{L^{\infty}(M)}\leq K$. There exists
$\varepsilon_{0}>0$ depending only on $K,M,N$ such that if
\[
E_{p}(u_{0})=\frac{1}{p}\int_{M}|\nabla_{b}u_{0}|^{p}d\mu \leq \varepsilon
_{0},\
\]
then there is a unique global weak solution $\varphi$ of (\ref{4}) with
$\partial_{t}\varphi \in L^{2}(M\times \lbrack0,\infty))$ and $\varphi
,\nabla_{b}\varphi \in C^{\beta}(M\times \lbrack0,\infty),N)$, where $0<\beta
<1$. Moreover, there exists a sequence $t_{k}\rightarrow \infty$ such that
$\varphi(t_{k})$ converges in $C^{1,\beta^{\prime}}(M,N)$ for all
$\beta^{\prime}<\beta$, to a weakly p-pseudoharmonic map $\varphi_{\infty}\in$
$C^{1,\beta}(M,N)$ satisfying
\[
E_{p}(\varphi_{\infty})\leq E_{p}(u_{0}).
\]
Moreover, there exists $\overline{\varepsilon_{0}}>0$ depending only on
$K,M,N\ $and$\ p$ such that if in addition $E_{p}(u_{0})\leq \overline
{\varepsilon_{0}}$, then $\varphi_{\infty}$ is a constant map.
\end{theorem}

\begin{remark}
1. One may also compare (\ref{4}) to the well-known $p$-harmonic map heat flow
from a closed Riemannian manifold $(M,h_{ij})$ to compact Riemannian manifold
$(N^{m},g_{ij})$\ :
\begin{equation}
\left \{
\begin{array}
[c]{l}%
\frac{\partial u}{\partial t}=\triangle_{p}^{M}u+|\nabla u|^{p-2}A(u)(\nabla
u,\nabla u),\\
u(\omega,0)=u_{0}(\omega),
\end{array}
\right.  \label{5}%
\end{equation}
where $A$ is the second fundamental form of $N$ in $R^{m+k}$. \ In the papers
of \cite{fr1} and \cite{fr2}, A. Fardoun and R. Regabaoui proved that if the
sectional curvature $K^{N}$ is nonpositive, then the heat flow (\ref{5}) has a
unique global weak solution $u$. Moreover, $u(t)$ converges to a $p$-harmonic
map $u_{\infty}$. Without the curvature assumption on the target manifold,
they showed that if $u_{0}\in C^{2,\alpha}(M,N)$, $0<\alpha<1$ with the small
$p$-energy for $\ p\leq \dim M,$ then there exists a unique global weak
solution $u$ of (\ref{5}). Moreover, as $t\rightarrow \infty$, $u(t)$ converges
to a constant map. \ For the further study of global weak solutions to the
$p$-harmonic map heat flow (\ref{5}), we refer to \cite{chh}, \cite{h1},
\cite{h2}, etc for some details. We also refer to \cite{hl} for the gradient
estimate of the minimizing $p$-harmonic map.

2. Since $(2n+2)$ is homogeneous dimnesion of a pseudohermitian manifold
$M^{2n+1},$ then when $p>2n+2,$ we do not need the smallness assumption in
Theorem \ref{t2}. By following the same steps, the result of Theorem \ref{t2}
holds for any smooth initial data. \ Thus we may assume that $p\leq2n+2$ in
the proof.

3. In the paper of \cite[Remark 5.3.]{cdry}, the second named author proved
that $f:M\rightarrow N$ is harmonic if and only if $f$ is pseudoharmonic
whenever $M$ is Sasakian and $N$ is a Riemannian manifold with nonpositive
curvature. It is interesting to know whether it is true for $p$-harmonic maps
and $p$-pseudoharmonic maps.
\end{remark}

\textbf{Acknowledgments. }The third named author would like to express his
thanks to Taida Institute for Mathematical Sciences (TIMS), National Taiwan
University. Part of the project was done during his visit to TIMS.

\section{Preliminaries}

We first introduce some basic materials on pseudohermitian $(2n+1)$-manifolds
(\cite{l}). Let $(M,\xi)$ be a $(2n+1)$-dimensional, orientable, contact
manifold with contact structure $\xi$, $dim_{R}\xi=2n$. A CR structure
compatible with $\xi$ is an endomorphism $J:\xi \rightarrow \xi$ such that
$J^{2}=-id$. We also assume that $J$ satisfies the following integrability
condition: if $X$ and $Y$ are in $\xi$, then so are $[JX,Y]+[X,JY]$ and
$J([JX,Y]+[X,JY])=[JX,JY]-[X,Y]$. A CR structure $J$ can extend to
$C\otimes \xi$, and decomposes $C\otimes \xi$ into the direct sum of $T_{1,0}M$
and $T_{0,1}M$, which are eigenspaces of $J$ with respect to eigenvalues $i$
and $-i$, respectively. A manifold $M$ with a CR structure is called a CR
manifold. A pseudohermitian structure compatible with $\xi$ is a CR structure
$J$ with $\xi$ together with a choice of contact form $\theta$. Such a choice
determines a unique real vector field $T$ transverse to $\xi$, which is called
the characteristic vector field of $\theta$, such that $\theta(T)=1$ and
$L_{T}\theta=0$. Let $\{T,Z_{\alpha},Z_{\bar{\alpha}}\}$ be a frame of
$TM\otimes C$, where $Z_{\alpha}$ is any local frame of $T_{1,0}M$,
$Z_{\bar{\alpha}}\in T_{0,1}M$, and $\mathbf{T}$ is the charecteristic vector
field. Then, $\{ \theta,\theta^{\alpha},\theta^{\bar{\theta}}\}$, which is the
coframe dual to $\{T,Z_{\alpha},Z_{\bar{\alpha}}\}$, satisfies $d\theta
=ih_{\alpha \bar{\beta}}\theta^{\alpha}\wedge \theta^{\bar{\beta}}$, for some
positive definite hermitian matrix of functions $(h_{\alpha \bar{\beta}})$.
Locally, one can choose $Z_{\alpha}$ appropriately so that $h_{\alpha
\bar{\beta}}=\delta_{\alpha \beta}$ to simplify tensorial calculation.

The Levi form $\langle,\rangle_{L_{\theta}}$ is the Hermitian form on
$T_{1,0}M$ defined by
\[
\langle Z,W\rangle_{L_{\theta}}=-i\langle d\theta,Z\wedge \bar{W}\rangle.
\]
We can extend $\langle,\rangle_{L_{\theta}}$ to $T_{0,1}M$ by defining
$\langle \bar{Z},\bar{W}\rangle_{L_{\theta}}=\bar{\langle,Z,W\rangle
_{L_{\theta}}}$ for all $Z,W\in T_{1,0}M$. The Levi form induces naturally a
Hermitian form on the dual bundle of $T_{1,0}M$, denoted by $\langle
,\rangle_{L_{\theta}^{\ast}}$, and hence on all the induced tensor bundles.
Integrating the Hermitian form over $M$ with respect to the volume form
$d\mu=\theta \wedge(d\theta)^{n}$, we get an inner product on the space of
sections of each tensor bundle. We denote the inner product by notation
$\langle,\rangle$. For example,
\[
\langle u,v\rangle=\int_{M}u\bar{v}d\mu,
\]
for functions $u$ and $v$.

The pseudohermitian connection of $(J,\theta)$ is the connection $\nabla$ on
$TM\otimes C$ (and extended to tensors) given in terms of local frames
$Z_{\alpha}\in T_{1,0}M$ by $\nabla Z=\omega_{\alpha}^{\beta}\otimes Z_{\beta
}$, $\nabla Z_{\bar{\alpha}}=\omega_{\bar{\alpha}}^{\bar{\beta}}\otimes
Z_{\bar{\beta}}$, $\nabla T=0$, where $\omega_{\alpha}^{\beta}$ are the
1-forms uniquely determined by the following equations:
\begin{align*}
d\theta^{\beta}  &  =\theta^{\alpha}\wedge \omega_{\alpha}^{\beta}+\theta
\wedge \tau^{\beta}\\
0  &  =\tau_{\alpha}\wedge \theta^{\alpha}\\
0  &  =\omega_{\alpha}^{\beta}+\omega_{\bar{\beta}}^{\bar{\alpha}}.
\end{align*}
We can write (by Cartan lemma) $\tau_{\alpha}=A_{\alpha \gamma}\theta^{\gamma}%
$, with $A_{\alpha \gamma}=A_{\gamma \alpha}$ the pseudohermitian torsion of
$(M,J,\theta)$. The curvature of this Tanaka-Webster connection, expressed in
terms of the coframe $\{ \theta=\theta^{0}=\theta,\theta^{\alpha},\theta
^{\bar{\alpha}}\}$, is
\begin{align*}
&  \Pi_{\beta}^{\alpha}=\overline{\Pi_{\bar{\beta}}^{\bar{\alpha}}}%
=d\omega_{\beta}^{\alpha}-\omega_{\beta}^{\gamma}\wedge \omega_{\gamma}%
^{\alpha},\\
&  \Pi_{0}^{\alpha}=\Pi_{\alpha}^{0}=\Pi_{0}^{\bar{\beta}}=\Pi_{\bar{\beta}%
}^{0}=\Pi_{0}^{0}=0.
\end{align*}
Webster showed that $\Pi_{\beta}^{\alpha}$ can be written
\[
\Pi_{\beta}^{\alpha}=R_{\beta \rho \bar{\sigma}}^{\alpha}\theta^{\rho}%
\wedge \theta^{\bar{\sigma}}+W_{\beta \rho}^{\alpha}\theta^{\sigma}\wedge
\theta-W_{\beta \bar{\rho}}^{\alpha}\theta^{\bar{\rho}}\wedge \theta
+i\theta_{\beta}\wedge \tau^{\alpha}-i\tau_{\beta \wedge \theta^{\alpha}},
\]
where the coefficients satisfy
\[
R_{\beta \bar{\alpha}\rho \bar{\sigma}}=\overline{R_{\alpha \bar{\beta}%
\bar{\sigma}\rho}}=R_{\bar{\alpha}\beta \bar{\sigma}\rho}=R_{\rho \bar{\alpha
}\beta \bar{\sigma}},\quad W_{\beta \bar{\alpha}\gamma}=A_{\beta \gamma
,\bar{\alpha}}.
\]

We will denote components of covariant derivatives with indices preceded by
comma; thus write $A_{\alpha \beta;\gamma}$. The indices $\{0,\alpha
,\bar{\alpha}\}$ indicate derivatives with respect to $\{T,Z_{\alpha}%
,Z_{\bar{\alpha}}\}$. For derivatives of a scalar function, we will often omit
the comma, for instance, $\varphi_{\alpha}=Z_{\alpha}\varphi$, $\varphi
_{\alpha \bar{\beta}}=Z_{\bar{\beta}}Z_{\alpha}\varphi-\omega_{\alpha}^{\gamma
}(Z_{\bar{\beta}})Z_{\gamma}\varphi$, $\varphi_{0}=\mathbf{T}\varphi$ for a
smooth function $\varphi$.

For a smooth real function $\varphi$, the subgradient $\nabla_{b}$ is defined
by $\nabla_{b}\varphi \in \xi$ and $<Z,\nabla_{b}\varphi \rangle_{L_{\theta}%
}=d\varphi(Z)$ for all vector fields $Z$ tangent to contact plane. Locally
$\nabla_{b}\varphi=\sum_{\alpha}(\varphi_{\bar{\alpha}}Z_{\alpha}%
+\varphi_{\alpha}Z_{\bar{\alpha}})$. We can use the connection to define the
subhessian as the complex linear map
\[
(\nabla^{H})^{2}\varphi:T_{1,0}\oplus T_{0,1}\rightarrow T_{1,0}\oplus T_{0,1}%
\]
by
\[
(\nabla^{H})^{2}\varphi(Z)=\nabla_{Z}\nabla_{b}\varphi.
\]
We also define the subdivergence operator $\operatorname{div}_{b}(\cdot)$ by
\[
\operatorname{div}_{b}(W)=W^{\beta},_{\beta}+W^{\overline{\beta}}%
,_{\overline{\beta}}%
\]
for all vector fields $W=W^{\beta}Z_{\beta}+W^{\overline{\beta}}%
Z_{\overline{\beta}}$. In particular
\[
|\nabla_{b}\varphi|^{2}=2\varphi_{\alpha}\varphi_{\bar{\alpha}},\quad
|\nabla_{b}^{2}\varphi|^{2}=2(\varphi_{\alpha \beta}\varphi_{\bar{\alpha}%
\bar{\beta}}+\varphi_{\alpha \bar{\beta}}\varphi_{\bar{\alpha}\beta})
\]
and
\[
\triangle_{b}\varphi=\operatorname{div}_{b}(\nabla_{b}\varphi)=(\varphi
_{\alpha \bar{\alpha}}+\varphi_{\bar{\alpha}\alpha}).
\]

We also recall below what the Folland-Stein space $S^{k,p}$ is. Let $D$ denote
a differential operator acting on functions. We say $D$ has weight $m$,
denoted $w(D)=m,$ if $m$ is the smallest integer such that $D$ can be locally
expressed as a polynomial of degree $m$ in vector fields tangent to the
contact bundle $\xi$. We define the Folland-Stein space $S^{k,p}$ of functions
on $M$ by
\[
S^{k,p}=\{f\in L^{p}:Df\in L^{p}\text{ \  \textrm{whenever} \ }w(D)\leq k\}.
\]
We define the $L^{p}$ norm of $\nabla_{b}f,$ $\nabla_{b}^{2}f$, ... to be
($\int|\nabla_{b}f|^{p}\theta \wedge(d\theta)^{n})^{1/p},$ ($\int|\nabla
_{b}^{2}f|^{p}\theta \wedge(d\theta)^{n})^{1/p},$ $...,$ respectively, as
usual. So it is natural to define the $S^{k,p}$ norm of $f\in S^{k,p}$ as
follows:
\[
||f||_{S^{k,p}}\equiv(\sum_{0\leq j\leq k}||\nabla_{b}^{j}f||_{L^{p}}%
^{p})^{1/p}.
\]
The function space $S^{k,p}$ with the above norm is a Banach space for
$k\geq0,$ $1<p<\infty.$ There are also embedding theorems of Sobolev type. The
reader can make reference to \cite{f} and \cite{fs} for more details of these spaces.

In this paper, we embed $N$ isometrically into the Euclidean space
$\mathbb{R}^{l}$ with $l$ large enough and then $S^{k,p}=S^{k,p}%
(M,\mathbb{R}^{l})$. Let $\pi:\mathbb{R}^{l}\rightarrow N$ \ be a smooth
projection. Define $S^{k,p}(M,N)$ by $\pi \big(S^{k,p}\big)$ (similarly,
$Lip(M,N):=\pi \big(Lip(M,\mathbb{R}^{l})\big)$ and so do other spaces of maps
from $M$ to $N$). From now on, the upper indices $j$'s of $\{ \varphi
^{j},d\varphi^{j},\cdots \}$ start from $1$ to $l$ if we do not specify them.

\section{Moser-type Harnack Inequality}

In this section, we derive the Moser-type Harnack inequality (\cite{m},
\cite{cc1}) for the total regularized energy density%
\[
g:=f_{\delta}+\varepsilon e_{0}%
\]
with
\[
f_{\delta}:=|\nabla_{b}\varphi|^{2}+\delta.
\]

Let $\varphi:(M,J,\theta)\rightarrow(N,g_{ij})$ be a map from $(M,J,\theta)$
to $(N,g_{ij})$. We first derive the Euler-Lagrange equation associated to its
$p$-energy $E_{p}(\varphi).$

\begin{lemma}
\label{l31} Let $(M,J,\theta)$ be a closed pseudohermitian manifold and
$(N^{m},g)$ be a Riemannian manifold. A $C^{2}$ map $\varphi:(M,J,\theta
)\rightarrow(N,g)$ is $p$-pseudoharmonic if and only if it satisfies the
Euler-Lagrangian equations
\begin{equation}
\triangle_{b,p}\varphi^{k}+2|\nabla_{b}\varphi|^{p-2}\widetilde{\Gamma}%
_{ij}^{k}\varphi_{\alpha}^{i}\varphi_{\bar{\alpha}}^{j}=0,\  \  \ k=1,...,m
\label{100}%
\end{equation}
where $\triangle_{b,p}\varphi^{k}=\operatorname{div}_{b}(|\nabla_{b}%
\varphi|^{p-2}\nabla_{b}\varphi^{k})$.
\end{lemma}

\begin{proof}
Let $\varphi_{t}$, $-\varepsilon<t<\varepsilon$, be a smooth variation of
$\varphi$ so that
\[
\varphi_{0}=\varphi \quad \text{and}\quad \frac{d\varphi_{t}}{dt}|_{t=0}%
=V\in \Gamma(\varphi^{-1}TN).
\]
$\varphi_{t}$ may be viewed as a map from $(-\varepsilon,\varepsilon)\times M$
into $N$. By direct computations, we have
\begin{align*}
&  \frac{dE_{p}(\varphi_{t})}{dt}=\frac{1}{p}\frac{d}{dt}\int_{M}|\nabla
_{b}\varphi_{t}|^{p}d\mu \\
&  =\frac{1}{2}\int_{M}|\nabla_{b}\varphi_{t}|^{p-2}\frac{d}{dt}|\nabla
_{b}\varphi_{t}|^{2}d\mu=\int_{M}|\nabla_{b}\varphi_{t}|^{p-2}\frac{d}%
{dt}[g_{ij}\varphi_{t\alpha}^{i}\varphi_{t\bar{\alpha}}^{j}]d\mu \\
&  =\int_{M}|\nabla_{b}\varphi_{t}|^{p-2}[g_{ij,k}\varphi_{t\alpha}^{i}%
\varphi_{t\bar{\alpha}}^{j}\frac{d\varphi_{t}^{k}}{dt}+g_{ij}(\frac
{d\varphi_{t}^{i}}{dt})_{t\alpha}\varphi_{t\bar{\alpha}}^{j}+g_{ij}%
(\frac{d\varphi_{t}^{j}}{dt})_{\bar{\alpha}}\varphi_{\alpha}^{i}]d\mu \\
&  =-\int_{M}[|\nabla_{b}\varphi_{t}|^{p-2}\triangle_{b}\varphi_{t}%
^{k}+\langle \nabla_{b}|\nabla_{b}\varphi_{t}|^{p-2},\nabla_{b}\varphi_{t}%
^{k}\rangle+2|\nabla_{b}\varphi_{t}|^{p-2}\widetilde{\Gamma}_{ij}^{k}%
\varphi_{t\alpha}^{i}\varphi_{t\bar{\alpha}}^{j}]\frac{d\varphi_{t}^{k}}%
{dt}d\mu \\
&  =-\int_{M}[\triangle_{b,p}\varphi_{t}^{k}+2|\nabla_{b}\varphi_{t}%
|^{p-2}\widetilde{\Gamma}_{ij}^{k}\varphi_{t\alpha}^{i}\varphi_{t\bar{\alpha}%
}^{j}]\frac{d\varphi_{t}^{k}}{dt}d\mu \\
&  =-\int_{M}\langle \frac{d\varphi_{t}}{dt},\tau_{p}(\varphi_{t})\rangle d\mu.
\end{align*}
Thus, the first variational formula is given by
\[
\frac{d}{dt}E_{p}(\varphi_{t})|_{t=0}=-\int_{M}\langle V,\tau_{p}%
(\varphi)\rangle d\mu,
\]
where $\tau_{p}(\varphi)$ is called the $p$-tension field of $\varphi$, which
is defined by
\[
\tau_{p}(\varphi)=\sum_{k=1}^{m}[\triangle_{b,p}\varphi^{k}+2|\nabla
_{b}\varphi|^{p-2}\widetilde{\Gamma}_{ij}^{k}\varphi_{\alpha}^{i}\varphi
_{\bar{\alpha}}^{j}]\frac{\partial}{\partial y_{k}}.
\]
Therefore $\varphi \in C^{2}(M;N)$ is a critical point of the $p$-energy
functional $E_{p}(\varphi)$ if and only if its $p$-tension field $\tau_{p}(u)$
vanishes identically. That is, $\varphi$ is $p$-pseudoharmonic if and only if
it satisfies the Euler-Lagrange equations (\ref{100}).
\end{proof}

Next we recall the CR version of Bochner formula for a real smooth function on
a closed pseudohermitian manifold $(M^{2n+1},J,\theta)$.

\begin{lemma}
\label{CR1}(\cite{g}) Let $(M^{2n+1},J,\theta)$ be a closed pseudohermitian
manifold. For a real smooth function $u$ on $(M,J,\theta)$,
\begin{equation}%
\begin{split}
\frac{1}{2}\Delta_{b}|\nabla_{b}u|^{2}  &  =|(\nabla^{H})^{2}u|^{2}%
+\langle \nabla_{b}u,\nabla_{b}\Delta_{b}u\rangle_{L_{\theta}}\\
&  \quad+[2Ric-(n-2)Tor]\big((\nabla_{b}u)_{\mathbf{C}},(\nabla_{b}%
u)_{\mathbf{C}}\big)+2\langle J\nabla_{b}u,\nabla_{b}u_{0}\rangle_{L_{\theta}%
}.
\end{split}
\label{CR}%
\end{equation}
Here $(\nabla_{b}u)_{\mathbf{C}}=u_{\overline{\alpha}}Z_{\alpha}$ is the
corresponding complex $(1,0)$-vector field of $\  \nabla_{b}u$ and
$d_{b}u=u_{\alpha}\theta^{\alpha}+u_{\overline{\alpha}}\theta^{\overline
{\alpha}}.$
\end{lemma}

Since $f_{\delta}$ and $e_{0}(\varphi)$ are independent of the choice of local
coordinates, for each point $(x,t)$ one may choose a normal coordinate chart
$U$ at $(x,t)$ and a normal coordinate chart at $\varphi(x,t)$ such that
$\varphi(U)\subset V$ and then fulfill the following computations at the point
$(x,t)$.

Now we are ready to derive the Moser type Harnack inequality for the total
regularized energy density.

\begin{lemma}
\label{l32} Let $(M^{2n+1},J,\theta)$ be a closed Sasakian manifold and
$(N,g_{ij})$ be a compact Riemannian manifold. The solution $\varphi$ of the
regularized equation (\ref{4a}) satisfying the following inequalities:

$(i)$ For $g=f_{\delta}+\varepsilon e_{0}=|\nabla_{b}\varphi|^{2}%
+\delta+\varepsilon e_{0},$
\begin{align}
\label{32A} &  \frac{\partial g}{\partial t}-\operatorname{div}_{b}(f_{\delta
}^{\frac{p-2}{2}}\nabla_{b}g)-(p-2)\operatorname{div}_{b}(f_{\delta}%
^{\frac{p-4}{2}}\langle \nabla_{b}f_{\delta},\nabla_{b}\varphi^{k}\rangle
\nabla_{b}\varphi^{k})\nonumber \\
&  -2\varepsilon \operatorname{div}_{b}((f_{\delta}^{\frac{p-2}{2}})_{0}%
\nabla_{b}\varphi^{k})\varphi_{0}^{k}+\frac{p-2}{2}f_{\delta}^{\frac{p-4}{2}%
}|\nabla_{b}f_{\delta}|^{2}+2f_{\delta}^{\frac{p-2}{2}}|\nabla_{b}^{2}%
\varphi^{k}|^{2}+2\varepsilon f_{\delta}^{\frac{p-2}{2}}|\nabla_{b}\varphi
_{0}^{k}|^{2}\\
&  \leq C(f_{\delta}^{\frac{p}{2}}+f_{\delta}^{\frac{p+2}{2}})+C\varepsilon
f_{\delta}^{\frac{p}{2}}e_{0}-4f_{\delta}^{\frac{p-2}{2}}\langle J\nabla
_{b}\varphi^{k},\nabla_{b}\varphi_{0}^{k}\rangle.\nonumber
\end{align}

$(ii)$ If \ the sectional curvature of $(N,g_{ij})$ is nonpositive
\[
K^{N}\leq0,
\]
then
\begin{align}
\label{32B} &  \frac{\partial g}{\partial t}-\operatorname{div}_{b}(f_{\delta
}^{\frac{p-2}{2}}\nabla_{b}g)-(p-2)\operatorname{div}_{b}(f_{\delta}%
^{\frac{p-4}{2}}\langle \nabla_{b}f_{\delta},\nabla_{b}\varphi^{k}\rangle
\nabla_{b}\varphi^{k})\nonumber \\
&  -2\varepsilon \operatorname{div}_{b}((f_{\delta}^{\frac{p-2}{2}})_{0}%
\nabla_{b}\varphi^{k})\varphi_{0}^{k}+\frac{p-2}{2}f_{\delta}^{\frac{p-4}{2}%
}|\nabla_{b}f_{\delta}|^{2}+2f_{\delta}^{\frac{p-2}{2}}|\nabla_{b}^{2}%
\varphi^{k}|^{2}+2\varepsilon f_{\delta}^{\frac{p-2}{2}}|\nabla_{b}\varphi
_{0}^{k}|^{2}\\
&  \leq Cf_{\delta}^{\frac{p}{2}}-4f_{\delta}^{\frac{p-2}{2}}\langle
J\nabla_{b}\varphi^{k},\nabla_{b}\varphi_{0}^{k}\rangle.\nonumber
\end{align}

\end{lemma}

By Young's inequality, the bad term $\langle J\nabla_{b}\varphi^{k},\nabla
_{b}\varphi_{0}^{k}\rangle$ in the RHS of (\ref{32A}) will be dominated by the
good term $|\nabla_{b}\varphi_{0}|^{2}$ in LHS. As a consequence, we have

\begin{corollary}
\label{c1} Let $(M^{2n+1},J,\theta)$ be a closed Sasakian manifold and
$(N,g_{ij})$ be a compact Riemannian manifold. The solution $\varphi$ of the
regularized equation (\ref{4a}) satisfying the following inequalities:

$(i)$ For $g=f_{\delta}+\varepsilon e_{0}=|\nabla_{b}\varphi|^{2}%
+\delta+\varepsilon e_{0},$
\begin{align*}
&  \frac{\partial g}{\partial t}-\operatorname{div}_{b}(f_{\delta}^{\frac
{p-2}{2}}\nabla_{b}g)-(p-2)\operatorname{div}_{b}(f_{\delta}^{\frac{p-4}{2}%
}\langle \nabla_{b}f_{\delta},\nabla_{b}\varphi^{k}\rangle \nabla_{b}\varphi
^{k}\\
&  -2\varepsilon \operatorname{div}_{b}((f_{\delta}^{\frac{p-2}{2}})_{0}%
\nabla_{b}\varphi^{k})\varphi_{0}^{k}+\frac{p-2}{2}f_{\delta}^{\frac{p-4}{2}%
}|\nabla_{b}f_{\delta}|^{2}+2f_{\delta}^{\frac{p-2}{2}}|\nabla_{b}^{2}%
\varphi^{k}|^{2}+\varepsilon f_{\delta}^{\frac{p-2}{2}}|\nabla_{b}\varphi
_{0}^{k}|^{2}\\
&  \leq C(f_{\delta}^{\frac{p}{2}}+f_{\delta}^{\frac{p+2}{2}})+C\varepsilon
f_{\delta}^{\frac{p}{2}}e_{0}.
\end{align*}

$(ii)$ If \ the sectional curvature of $(N,g_{ij})$ is nonpositive
\[
K^{N}\leq0,
\]
then
\begin{align*}
&  \frac{\partial g}{\partial t}-\operatorname{div}_{b}(f_{\delta}^{\frac
{p-2}{2}}\nabla_{b}g)-(p-2)\operatorname{div}_{b}(f_{\delta}^{\frac{p-4}{2}%
}\langle \nabla_{b}f_{\delta},\nabla_{b}\varphi^{k}\rangle \nabla_{b}\varphi
^{k})\\
&  -2\varepsilon \operatorname{div}_{b}((f_{\delta}^{\frac{p-2}{2}})_{0}%
\nabla_{b}\varphi^{k})\varphi_{0}^{k}+\frac{p-2}{2}f_{\delta}^{\frac{p-4}{2}%
}|\nabla_{b}f_{\delta}|^{2}+2f_{\delta}^{\frac{p-2}{2}}|\nabla_{b}^{2}%
\varphi|^{2}+\varepsilon f_{\delta}^{\frac{p-2}{2}}|\nabla_{b}\varphi_{0}%
|^{2}\\
&  \leq Cf_{\delta}^{\frac{p}{2}}.
\end{align*}

\end{corollary}

\begin{proof}
$(i)$ We first compute $\frac{\partial}{\partial t}f_{\delta}%
-\operatorname{div}_{b}(f_{\delta}^{\frac{p-2}{2}}\nabla_{b}f_{\delta})$ :

Note that $|\nabla_{b}\varphi|^{2}=2g_{ij}\varphi_{\alpha}^{i}\varphi
_{\bar{\alpha}}^{j}$. It is straightforward to compute as
\[%
\begin{array}
[c]{ccl}%
\frac{\partial f_{\delta}}{\partial t} & = & \frac{\partial}{\partial
t}(2g_{ij}\varphi_{\alpha}^{i}\varphi_{\bar{\alpha}}^{j}+\delta)\\
& = & 2g_{ij}(\frac{\partial \varphi^{i}}{\partial t})_{\alpha}\varphi
_{\bar{\alpha}}^{j}+2g_{ij}(\frac{\partial \varphi^{j}}{\partial t}%
)_{\bar{\alpha}}\varphi_{\alpha}^{i}\\
& = & 2g_{kl}[div(f_{\delta}^{\frac{p-2}{2}}\nabla_{b}\varphi^{k})+2f_{\delta
}^{\frac{p-2}{2}}\tilde{\Gamma}_{ij}^{k}\varphi_{\alpha}^{i}\varphi
_{\bar{\alpha}}^{j}]_{\beta}\varphi_{\bar{\beta}}^{l}\\
&  & +2g_{kl}[div(f_{\delta}^{\frac{p-2}{2}}\nabla_{b}\varphi^{k})+2f_{\delta
}^{\frac{p-2}{2}}\tilde{\Gamma}_{ij}^{k}\varphi_{\alpha}^{i}\varphi
_{\bar{\alpha}}^{j}]_{\bar{\beta}}\varphi_{\beta}^{l}\\
& = & 2\langle \nabla_{b}\varphi^{k},\nabla_{b}div(f_{\delta}^{\frac{p-2}{2}%
}\nabla_{b}\varphi^{k})\rangle \\
&  & +2f_{\delta}^{\frac{p-2}{2}}[2\widetilde{\Gamma}_{ij,l}^{k}\varphi
_{\beta}^{l}\varphi_{\alpha}^{i}\varphi_{\bar{\alpha}}^{j}\varphi_{\bar{\beta
}}^{k}+2\widetilde{\Gamma}_{ij,l}^{k}\varphi_{\bar{\beta}}^{l}\varphi_{\alpha
}^{i}\varphi_{\bar{\alpha}}^{j}\varphi_{\beta}^{k}]
\end{array}
\]
and from the CR Bochner formula (\ref{CR})
\begin{align*}
&  \operatorname{div}_{b}(f_{\delta}^{\frac{p-2}{2}}\nabla_{b}f_{\delta})\\
&  =\frac{p-2}{2}f_{\delta}^{^{\frac{p-4}{2}}}|\nabla_{b}f_{\delta}%
|^{2}+f_{\delta}^{\frac{p-2}{2}}\triangle_{b}f_{\delta}\\
&  =\frac{p-2}{2}f_{\delta}^{^{\frac{p-4}{2}}}|\nabla_{b}f_{\delta}%
|^{2}+2f_{\delta}^{\frac{p-2}{2}}\triangle_{b}[g_{ij}\varphi_{\alpha}%
^{i}\varphi_{\bar{\alpha}}^{j}]\\
&  =\frac{p-2}{2}f_{\delta}^{^{\frac{p-4}{2}}}|\nabla_{b}f_{\delta}%
|^{2}+2f_{\delta}^{\frac{p-2}{2}}[\frac{1}{2}\triangle_{b}|\nabla_{b}%
\varphi^{k}|^{2}+\varphi_{\alpha}^{i}\varphi_{\bar{\alpha}}\triangle_{b}%
g_{ij}]\\
&  =\frac{p-2}{2}f_{\delta}^{^{\frac{p-4}{2}}}|\nabla_{b}f_{\delta}%
|^{2}+2f_{\delta}^{\frac{p-2}{2}}[|\nabla_{b}^{2}\varphi^{k}|^{2}%
+\langle \nabla_{b}\varphi^{k},\nabla_{b}\triangle_{b}\varphi^{k}%
\rangle+\varphi_{\alpha}^{i}\varphi_{\bar{\alpha}}\triangle_{b}g_{ij}\\
&  \  \  \ +(2Ric-(n-2)Tor)((\nabla_{b}\varphi^{k})_{C},(\nabla_{b}\varphi
^{k})_{C})+2\langle J\nabla_{b}\varphi^{k},\nabla_{b}\varphi_{0}^{k}\rangle]
\end{align*}
Thus, based on \cite[(3.2) and (3.3)]{cc1}
\begin{align}
&  \frac{\partial f_{\delta}}{\partial t}-\operatorname{div}_{b}(f_{\delta
}^{\frac{p-2}{2}}\nabla_{b}f_{\delta})-(p-2)\operatorname{div}_{b}(f_{\delta
}^{\frac{p-4}{2}}\langle \nabla_{b}f_{\delta},\nabla_{b}\varphi^{k}%
\rangle \nabla_{b}\varphi^{k})\nonumber \\
&  +\frac{p-2}{2}f_{\delta}^{\frac{p-4}{2}}|\nabla_{b}f_{\delta}%
|^{2}+2f_{\delta}^{\frac{p-2}{2}}|\nabla_{b}^{2}\varphi|^{2}\label{32c}\\
&  =2f_{\delta}^{\frac{p-2}{2}}[2\widetilde{\Gamma}_{ij,l}^{k}\varphi_{\beta
}^{l}\varphi_{\alpha}^{i}\varphi_{\bar{\alpha}}^{j}\varphi_{\bar{\beta}}%
^{k}+2\widetilde{\Gamma}_{ij,l}^{k}\varphi_{\bar{\beta}}^{l}\varphi_{\alpha
}^{i}\varphi_{\bar{\alpha}}^{j}\varphi_{\beta}^{k}-\varphi_{\alpha}^{i}%
\varphi_{\bar{\alpha}}\triangle_{b}g_{ij}\nonumber \\
&  -(2Ric-(n-2)Tor)((\nabla_{b}\varphi^{k})_{C},(\nabla_{b}\varphi^{k}%
)_{C})]-4f_{\delta}^{\frac{p-2}{2}}\langle J\nabla_{b}\varphi^{k},\nabla
_{b}\varphi_{0}^{k}\rangle \nonumber \\
&  =2f_{\delta}^{\frac{p-2}{2}}[2\tilde{R}_{ijkl}\varphi_{\alpha}^{i}%
\varphi_{\beta}^{j}\varphi_{\bar{\alpha}}^{k}\varphi_{\bar{\beta}}^{l}%
+2\tilde{R}_{ijkl}\varphi_{\alpha}^{i}\varphi_{\bar{\beta}}^{j}\varphi
_{\bar{\alpha}}^{k}\varphi_{\beta}^{l}\nonumber
\end{align}%
\begin{align}
&  -(2Ric-(n-2)Tor)((\nabla_{b}\varphi^{k})_{C},(\nabla_{b}\varphi^{k}%
)_{C})]]-4f_{\delta}^{\frac{p-2}{2}}\langle J\nabla_{b}\varphi^{k},\nabla
_{b}\varphi_{0}^{k}\rangle \nonumber \\
&  \leq Cf_{\delta}^{\frac{p}{2}}+Cf_{\delta}^{\frac{p+2}{2}}-4f_{\delta
}^{\frac{p-2}{2}}\langle J\nabla_{b}\varphi^{k},\nabla_{b}\varphi_{0}%
^{k}\rangle.\nonumber
\end{align}
This implies
\begin{align}
&  \frac{\partial f_{\delta}}{\partial t}-\operatorname{div}_{b}(f_{\delta
}^{\frac{p-2}{2}}\nabla_{b}f_{\delta})-(p-2)\operatorname{div}_{b}(f_{\delta
}^{\frac{p-4}{2}}\langle \nabla_{b}f_{\delta},\nabla_{b}\varphi^{k}%
\rangle \nabla_{b}\varphi^{k})\nonumber \\
&  +\frac{p-2}{2}f_{\delta}^{\frac{p-4}{2}}|\nabla_{b}f_{\delta}%
|^{2}+2f_{\delta}^{\frac{p-2}{2}}|\nabla_{b}^{2}\varphi^{k}|^{2}\label{32a}\\
&  \leq Cf_{\delta}^{\frac{p}{2}}+Cf_{\delta}^{\frac{p+2}{2}}-4f_{\delta
}^{\frac{p-2}{2}}\langle J\nabla_{b}\varphi^{k},\nabla_{b}\varphi_{0}%
^{k}\rangle.\nonumber
\end{align}
Next we compute $\frac{\partial}{\partial t}e_{0}(\varphi)-\operatorname{div}%
_{b}(f_{\delta}^{\frac{p-2}{2}}\nabla_{b}e_{0}(\varphi))$ :

We observe from (\cite{cc2}) that for any smooth function $u$,
\[
\lbrack \Delta_{b},T]u=4\Big[i\sum_{\alpha,\beta=1}^{n}(A_{\bar{\alpha}%
\bar{\beta}}u_{\beta})_{\alpha}\Big].
\]
Thus for any Sasakian manifold $(M^{2n+1},J,\theta)$ (i.e. vanishing
pseudohermitian torsion),
\begin{equation}
\lbrack \Delta_{b},T]u=0. \label{32aa}%
\end{equation}

We first compute
\[
\frac{\partial e_{0}(\varphi)}{\partial t}=2g_{ij}(\frac{\partial \varphi^{i}%
}{\partial t})_{0}\varphi_{0}^{j}=2g_{kl}[\operatorname{div}_{b}(f_{\delta
}^{\frac{p-2}{2}}\nabla_{b}\varphi^{k})+2f_{\delta}^{\frac{p-2}{2}}%
\tilde{\Gamma}_{ij}^{k}\varphi_{\alpha}^{i}\varphi_{\bar{\alpha}}^{j}%
]_{0}\varphi_{0}^{l}%
\]
and
\[
\operatorname{div}_{b}(f_{\delta}^{\frac{p-2}{2}}\nabla_{b}e_{0}%
(\varphi))=\langle \nabla_{b}f_{\delta}^{\frac{p-2}{2}},\nabla_{b}e_{0}%
\rangle+f_{\delta}^{\frac{p-2}{2}}[2\varphi_{0}^{i}\triangle_{b}\varphi
_{0}^{i}+2|\nabla_{b}\varphi_{0}^{i}|^{2}+\varphi_{0}^{i}\varphi_{0}%
^{j}\triangle_{b}g_{ij}].
\]
Based on \cite[(3.4)]{cc1} and (\ref{32aa}), one can derive
\begin{align*}
&  \frac{\partial}{\partial t}e_{0}(\varphi)-\operatorname{div}_{b}(f_{\delta
}^{\frac{p-2}{2}}\nabla_{b}e_{0}(\varphi))-2\operatorname{div}_{b}((f_{\delta
}^{\frac{p-2}{2}})_{0}\nabla_{b}\varphi^{k})\varphi_{0}^{k}\\
&  =f_{\delta}^{\frac{p-2}{2}}[4\widetilde{\Gamma}_{ij,l}^{k}\varphi_{\alpha
}^{i}\varphi_{\bar{\alpha}}^{j}\varphi_{0}^{k}\varphi_{0}^{l}-\varphi_{0}%
^{i}\varphi_{0}^{j}\triangle_{b}g_{ij}]-2f_{\delta}^{\frac{p-2}{2}}|\nabla
_{b}\varphi_{0}|^{2}.
\end{align*}
On the other hand ( \cite[(3.5)]{cc1}),
\[
4\sum_{i,j,k,\ell=1}^{m}\sum_{\alpha=1}^{n}\widetilde{\Gamma}_{ij,l}%
^{k}\varphi_{\alpha}^{i}\varphi_{\overline{\alpha}}^{j}\varphi_{0}^{k}%
\varphi_{0}^{\ell}-\sum_{i,j=1}^{m}\varphi_{0}^{i}\varphi_{0}^{j}\Delta
_{b}(g_{ij})=4\sum_{i,j,k,\ell=1}^{m}\sum_{\alpha=1}^{n}\widetilde{R}%
_{ijk\ell}\varphi_{\alpha}^{i}\varphi_{0}^{j}\varphi_{\overline{\alpha}}%
^{k}\varphi_{0}^{\ell}.
\]
Hence
\begin{align}
\label{32b} &  \frac{\partial}{\partial t}e_{0}(\varphi)-\operatorname{div}%
_{b}(f_{\delta}^{\frac{p-2}{2}}\nabla_{b}e_{0}(\varphi))-2\operatorname{div}%
_{b}((f_{\delta}^{\frac{p-2}{2}})_{0}\nabla_{b}\varphi^{k})\varphi_{0}%
^{k}\nonumber \\
&  =f_{\delta}^{\frac{p-2}{2}}[4\widetilde{R}_{ijk\ell}\varphi_{\alpha}%
^{i}\varphi_{0}^{j}\varphi_{\overline{\alpha}}^{k}\varphi_{0}^{\ell
}]-2f_{\delta}^{\frac{p-2}{2}}|\nabla_{b}\varphi_{0}|^{2}\\
&  \leq Cf_{\delta}^{^{\frac{p}{2}}}e_{0}-2f_{\delta}^{\frac{p-2}{2}}%
|\nabla_{b}\varphi_{0}|^{2}.\nonumber
\end{align}
From (\ref{32a}) and (\ref{32b}), we have
\begin{align*}
&  \frac{\partial g}{\partial t}-\operatorname{div}_{b}(f_{\delta}^{\frac
{p-2}{2}}\nabla_{b}g)-(p-2)\operatorname{div}_{b}(f_{\delta}^{\frac{p-4}{2}%
}\langle \nabla_{b}f_{\delta},\nabla_{b}\varphi^{k}\rangle \nabla_{b}\varphi
^{k})\\
&  -2\varepsilon \operatorname{div}_{b}((f_{\delta}^{\frac{p-2}{2}})_{0}%
\nabla_{b}\varphi^{k})\varphi_{0}^{k}+\frac{p-2}{2}f_{\delta}^{\frac{p-4}{2}%
}|\nabla_{b}f_{\delta}|^{2}+2f_{\delta}^{\frac{p-2}{2}}|\nabla_{b}^{2}%
\varphi|^{2}+2\varepsilon f_{\delta}^{\frac{p-2}{2}}|\nabla_{b}\varphi
_{0}|^{2}\\
&  \leq C(f_{\delta}^{\frac{p}{2}}+f_{\delta}^{\frac{p+2}{2}})+C\varepsilon
f_{\delta}^{\frac{p}{2}}e_{0}-4f_{\delta}^{\frac{p-2}{2}}\langle J\nabla
_{b}\varphi^{k},\nabla_{b}\varphi_{0}^{k}\rangle.
\end{align*}
$(ii)$ However, if \ the sectional curvature of $(N,g_{ij})$ is nonpositive
\[
K^{N}\leq0,
\]
it follows from (\ref{32c}) and (\ref{32b}) that%
\begin{align}
&  \frac{\partial f_{\delta}}{\partial t}-\operatorname{div}_{b}(f_{\delta
}^{\frac{p-2}{2}}\nabla_{b}f_{\delta})-(p-2)\operatorname{div}_{b}(f_{\delta
}^{\frac{p-4}{2}}\langle \nabla_{b}f_{\delta},\nabla_{b}\varphi^{k}%
\rangle \nabla_{b}\varphi^{k})\nonumber \\
&  +\frac{p-2}{2}f_{\delta}^{\frac{p-4}{2}}|\nabla_{b}f_{\delta}%
|^{2}+2f_{\delta}^{\frac{p-2}{2}}|\nabla_{b}^{2}\varphi^{k}|^{2}\\
&  \leq Cf_{\delta}^{\frac{p}{2}}-4f_{\delta}^{\frac{p-2}{2}}\langle
J\nabla_{b}\varphi^{k},\nabla_{b}\varphi_{0}^{k}\rangle.\nonumber
\end{align}
and%
\begin{align}
&  \frac{\partial}{\partial t}e_{0}(\varphi)-\operatorname{div}_{b}(f_{\delta
}^{\frac{p-2}{2}}\nabla_{b}e_{0}(\varphi))-2\operatorname{div}_{b}((f_{\delta
}^{\frac{p-2}{2}})_{0}\nabla_{b}\varphi^{k})\varphi_{0}^{k}\\
&  \leq-2f_{\delta}^{\frac{p-2}{2}}|\nabla_{b}\varphi_{0}|^{2}.\nonumber
\end{align}
All these imply%
\begin{align*}
&  \frac{\partial g}{\partial t}-\operatorname{div}_{b}(f_{\delta}^{\frac
{p-2}{2}}\nabla_{b}g)-(p-2)\operatorname{div}_{b}(f_{\delta}^{\frac{p-4}{2}%
}\langle \nabla_{b}f_{\delta},\nabla_{b}\varphi^{k}\rangle \nabla_{b}\varphi
^{k})\\
&  -2\varepsilon \operatorname{div}_{b}((f_{\delta}^{\frac{p-2}{2}})_{0}%
\nabla_{b}\varphi^{k})\varphi_{0}^{k}+\frac{p-2}{2}f_{\delta}^{\frac{p-4}{2}%
}|\nabla_{b}f_{\delta}|^{2}+2f_{\delta}^{\frac{p-2}{2}}|\nabla_{b}^{2}%
\varphi|^{2}+2\varepsilon f_{\delta}^{\frac{p-2}{2}}|\nabla_{b}\varphi
_{0}|^{2}\\
&  \leq Cf_{\delta}^{\frac{p}{2}}-4f_{\delta}^{\frac{p-2}{2}}\langle
J\nabla_{b}\varphi^{k},\nabla_{b}\varphi_{0}^{k}\rangle.
\end{align*}

This completes the proof of this lemma.
\end{proof}

\section{Proof of Main Results}

In this section, we will prove a uniform estimate for the $p$-energy density
and then the global existence and asymptotic convergence of the $p$%
-pseudoharmonic map heat flow.

\begin{lemma}
\label{l33} Let $(M^{2n+1},J,\theta)$ be a closed Sasakian manifold and
$(N,g_{ij})$ be a compact Riemannian manifold. Let $u_{0}\in C^{2,\alpha
}(M,N)$, $0<\alpha<1,$ $||\nabla u_{0}||_{L^{\infty}(M)}\leq K$ and
$\varphi_{\delta}$ is the solution of the regularized equation (\ref{4a}).

$(i)$\ For $g=|\nabla_{b}\varphi_{\delta}|^{2}+\delta+\varepsilon e_{0}$ and
all $q\geq \frac{p}{2}$, there exists $\varepsilon_{1}>0$ depending on $M,N$
and $q$ such that if
\begin{equation}
\sup_{0\leq t<T^{\prime}}||g(t,\cdot)||_{L^{n+1}(M)}\leq \varepsilon_{1},
\label{2016}%
\end{equation}
then
\begin{equation}
\sup_{1\leq t_{1}\leq T^{\prime}}||g||_{L^{q}([t_{1}-1,t_{1}]\times M)}\leq
C,\quad \text{when}\quad T^{\prime}>1 \label{33a}%
\end{equation}
and
\begin{equation}
||g||_{L^{q}([0,T^{\prime})\times M)}\leq C,\quad \text{when}\quad T^{\prime
}\leq1, \label{33b}%
\end{equation}
where $C$ is a constant depending on $K,M,N$ and $q$.

$(ii)$ In addition, if \ the sectional curvature of $(N,g_{ij})$ is
nonpositive
\[
K^{N}\leq0,
\]
then (\ref{33a}) and (\ref{33b}) hold without assumption (\ref{2016}).
\end{lemma}

\begin{proof}
$(i)$\ Let us first prove (\ref{33a}). Fix $t_{0}\geq1$ and $x_{0}\in M$. Let
$\rho<R\leq R_{0}=\inf(R_{M},1)$ where $R_{M}$ denotes the radius of ball on
which CR Sobolev inequality (\ref{sobo}) holds and set $Q_{R}=(t_{0}%
-R,t_{0})\times B(x_{0},R).$ We choose $\psi \in C_{0}(B(x_{0},R))$ such that
$\psi=1,\ B(x_{0},\rho);\  \ 0\leq \psi \leq1;\ |\nabla_{b}\psi|\leq
c(R-\rho)^{-1},$ and let $\eta \in C^{\infty}(R)$ with $\eta(t)=0,\  \ t\leq
t_{0}-R;\  \eta(t)=1,\ t\geq t_{0}-\rho;\ 0\leq \eta(t)\leq1;\  \ |\eta^{\prime
}|\leq c(R-\rho)^{-1}.$ We set $\phi(x,t)=\psi(x)\eta(t).$ Multiply inequality
(\ref{32A}) by $g^{r}\phi^{k}$, $r\geq0$, $k\in N$, and integrate on
$[t_{0}-R,t]\times B(x_{0},R),\ t_{0}-R<t\leq t_{0}$, \ we have
\begin{align}
&  \frac{1}{r+1}\sup_{t\leq t_{0}}\int_{B_{x_{0}}(R)}g^{r+1}\phi^{k}d\mu
+r\int_{Q_{R}}f_{\delta}^{\frac{p-2}{2}}g^{r-1}\phi^{k}|\nabla_{b}g|^{2}d\mu
dt\nonumber \label{101}\\
&  +(p-2)r\int_{Q_{R}}f_{\delta}^{\frac{p-4}{2}}\langle \nabla_{b}f_{\delta
},\nabla_{b}\varphi^{l}\rangle^{2}g^{r-1}\phi^{k}d\mu dt+(p-2)\varepsilon
\int_{Q_{R}}f_{\delta}^{\frac{p-4}{2}}(f_{\delta})_{0}^{2}g^{r}\phi^{k}d\mu
dt\nonumber \\
&  +\frac{p-2}{2}\int_{Q_{R}}f_{\delta}^{\frac{p-4}{2}}g^{r}\phi^{k}%
|\nabla_{b}f_{\delta}|^{2}d\mu dt+2\int_{Q_{R}}f_{\delta}^{\frac{p-2}{2}%
}|\nabla_{b}^{2}\varphi^{l}|^{2}g^{r}\phi^{k}d\mu dt\nonumber \\
&  +2\varepsilon \int_{Q_{R}}f_{\delta}^{\frac{p-2}{2}}g^{r}\phi^{k}|\nabla
_{b}\varphi_{0}^{l}|^{2}d\mu dt\nonumber \\
&  \leq \frac{k}{1+r}\int_{Q_{R}}g^{r+1}\phi^{k-1}|\frac{\partial \phi}{\partial
t}|d\mu dt-k\int_{Q_{R}}f_{\delta}^{\frac{p-2}{2}}g^{r}\phi^{k-1}\langle
\nabla_{b}g,\nabla_{b}\phi \rangle d\mu dt\\
&  -(p-2)k\int_{Q_{R}}f_{\delta}^{\frac{p-4}{2}}\langle \nabla_{b}f_{\delta
},\nabla_{b}\varphi^{l}\rangle \langle \nabla_{b}\varphi^{l},\nabla_{b}%
\phi \rangle \phi^{k-1}g^{r}d\mu dt\nonumber \\
&  -(p-2)k\varepsilon \int_{Q_{R}}f_{\delta}^{\frac{p-4}{2}}(f_{\delta}%
)_{0}\varphi_{0}^{l}g^{r}\phi^{k-1}\langle \nabla_{b}\varphi^{l},\nabla_{b}%
\phi \rangle d\mu dt\nonumber \\
&  -(p-2)r\varepsilon \int_{Q_{R}}f_{\delta}^{\frac{p-4}{2}}(f_{\delta}%
)_{0}g^{r-1}\varphi_{0}^{l}\phi^{k}\langle \nabla_{b}\varphi^{l},\nabla
_{b}g\rangle d\mu dt\nonumber \\
&  -(p-2)r\varepsilon \int_{Q_{R}}f_{\delta}^{\frac{p-4}{2}}\langle \nabla
_{b}f_{\delta},\nabla_{b}\varphi^{l}\rangle \langle \nabla_{b}\varphi^{l}%
,\nabla_{b}e_{0}\rangle g^{r-1}\phi^{k}d\mu dt\nonumber \\
&  +C\int_{Q_{R}}(f_{\delta}^{\frac{p}{2}}+f_{\delta}^{\frac{p+2}{2}%
}+\varepsilon f_{\delta}^{\frac{p}{2}}e_{0})g^{r}\phi^{k}d\mu dt-4\int_{Q_{R}%
}f_{\delta}^{\frac{p-2}{2}}g^{r}\phi^{k}\langle J\nabla_{b}\varphi^{l}%
,\nabla_{b}\varphi_{0}^{l}\rangle d\mu dt\nonumber
\end{align}
By using Young's inequalities, we have
\begin{align}
&  -(p-2)r\varepsilon \int_{Q_{R}}f_{\delta}^{\frac{p-4}{2}}\langle \nabla
_{b}f_{\delta},\nabla_{b}\varphi^{l}\rangle \langle \nabla_{b}\varphi^{l}%
,\nabla_{b}e_{0}\rangle g^{r-1}\phi^{k}d\mu dt\nonumber \label{102}\\
&  \leq \frac{(p-2)r}{4}\int_{Q_{R}}f_{\delta}^{\frac{p-4}{2}}\langle \nabla
_{b}f_{\delta},\nabla_{b}\varphi^{l}\rangle^{2}g^{r-1}\phi^{k}d\mu \nonumber \\
&  +(p-2)r\varepsilon^{2}\int_{Q_{R}}f_{\delta}^{\frac{p-4}{2}}\langle
\nabla_{b}\varphi^{l},\nabla_{b}e_{0}\rangle^{2}g^{r-1}\phi^{k}d\mu dt\\
&  \leq \frac{(p-2)r}{4}\int_{Q_{R}}f_{\delta}^{\frac{p-4}{2}}\langle \nabla
_{b}f_{\delta},\nabla_{b}\varphi^{l}\rangle^{2}g^{r-1}\phi^{k}d\mu \nonumber \\
&  +4(p-2)r\varepsilon^{2}\int_{Q_{R}}f_{\delta}^{\frac{p-2}{2}}|\nabla
_{b}\varphi_{0}^{l}|^{2}e_{0}g^{r-1}\phi^{k}d\mu dt\nonumber
\end{align}
and
\begin{align}
&  -(p-2)r\varepsilon \int_{Q_{R}}f_{\delta}^{\frac{p-4}{2}}(f_{\delta}%
)_{0}g^{r-1}\varphi_{0}^{l}\phi^{k}\langle \nabla_{b}\varphi^{l},\nabla
_{b}g\rangle d\mu dt\nonumber \label{103}\\
&  =-(p-2)r\varepsilon \int_{Q_{R}}f_{\delta}^{\frac{p-4}{2}}(f_{\delta}%
)_{0}g^{r-1}\varphi_{0}^{l}\langle \nabla_{b}\varphi^{l},\nabla_{b}f_{\delta
}\rangle \phi^{k}d\mu dt\nonumber \\
&  -(p-2)r\varepsilon \int_{Q_{R}}f_{\delta}^{\frac{p-4}{2}}(f_{\delta}%
)_{0}g^{r-1}\varphi_{0}^{l}\langle \nabla_{b}\varphi^{l},\varepsilon \nabla
_{b}e_{0}\rangle \phi^{k}d\mu dt\nonumber \\
&  \leq \frac{(p-2)\varepsilon}{4}\int_{Q_{R}}f_{\delta}^{\frac{p-4}{2}%
}(f_{\delta})_{0}^{2}g^{r}\phi^{k}d\mu+(p-2)r^{2}\varepsilon \int_{Q_{R}%
}f_{\delta}^{\frac{p-4}{2}}g^{r-2}e_{0}\langle \nabla_{b}\varphi^{l},\nabla
_{b}f_{\delta}\rangle^{2}\phi^{k}d\mu dt\\
&  +\frac{(p-2)\varepsilon}{4}\int_{M}f_{\delta}^{\frac{p-4}{2}}(f_{\delta
})_{0}^{2}g^{r}\phi^{k}d\mu+(p-2)r^{2}\varepsilon^{3}\int_{M}f_{\delta}%
^{\frac{p-4}{2}}g^{r-2}e_{0}f_{\delta}|\nabla_{b}\varphi_{0}^{l}|^{2}e_{0}%
^{2}\phi^{k}d\mu dt\nonumber \\
&  =\frac{(p-2)\varepsilon}{2}\int_{M}f_{\delta}^{\frac{p-4}{2}}(f_{\delta
})_{0}^{2}g^{r}\phi^{k}d\mu+(p-2)r^{2}\varepsilon \int_{Q_{R}}f_{\delta}%
^{\frac{p-4}{2}}g^{r-2}e_{0}\langle \nabla_{b}\varphi^{l},\nabla_{b}f_{\delta
}\rangle^{2}\phi^{k}d\mu dt\nonumber \\
&  +(p-2)r^{2}\varepsilon^{3}\int_{M}f_{\delta}^{\frac{p-4}{2}}g^{r-2}%
e_{0}f_{\delta}|\nabla_{b}\varphi_{0}^{l}|^{2}e_{0}^{2}\phi^{k}d\mu
dt\nonumber
\end{align}
and
\begin{align}
&  -(p-2)k\varepsilon \int_{Q_{R}}f_{\delta}^{\frac{p-4}{2}}(f_{\delta}%
)_{0}\varphi_{0}^{l}g^{r}\phi^{k-1}\langle \nabla_{b}\varphi^{l},\nabla_{b}%
\phi \rangle d\mu dt\label{104}\\
&  \leq \frac{(p-2)\varepsilon}{4}\int_{Q_{R}}f_{\delta}^{\frac{p-4}{2}%
}(f_{\delta})_{0}^{2}g^{r}\phi^{k}d\mu+(p-2)k^{2}\varepsilon \int_{Q_{R}%
}f_{\delta}^{\frac{p-4}{2}}g^{r}e_{0}f_{\delta}|\nabla_{b}\phi|^{2}\phi
^{k-2}d\mu dt\nonumber
\end{align}
and
\begin{align}
&  -(p-2)k\int_{Q_{R}}f_{\delta}^{\frac{p-4}{2}}\langle \nabla_{b}f_{\delta
},\nabla_{b}\varphi^{l}\rangle \langle \nabla_{b}\varphi^{l},\nabla_{b}%
\phi \rangle \phi^{k-1}g^{r}d\mu dt\nonumber \label{105}\\
&  \leq(p-2)\int_{Q_{R}}f_{\delta}^{\frac{p-4}{2}}g^{r-1}\langle \nabla
_{b}f_{\delta},\nabla_{b}\varphi^{l}\rangle^{2}\phi^{k}d\mu dt\\
&  +(p-2)\frac{k^{2}}{4}\int_{Q_{R}}f_{\delta}^{\frac{p-4}{2}}g^{r+1}%
f_{\delta}|\nabla_{b}\phi|^{2}\phi^{k-2}d\mu dt\nonumber
\end{align}
and
\begin{align}
&  -k\int_{Q_{R}}f_{\delta}^{\frac{p-2}{2}}g^{r}\phi^{k-1}\langle \nabla
_{b}g,\nabla_{b}\phi \rangle d\mu dt\nonumber \label{106}\\
&  \leq \frac{r}{2}\int_{Q_{R}}f_{\delta}^{\frac{p-2}{2}}g^{r-1}|\nabla
_{b}g|^{2}\phi^{k}d\mu dt+\frac{2k^{2}}{r}\int_{Q_{R}}\int_{Q_{R}}f_{\delta
}^{\frac{p-2}{2}}g^{r+1}|\nabla_{b}\phi|^{2}\phi^{k-2}d\mu dt
\end{align}
and
\begin{align}
&  -4\int_{Q_{R}}f_{\delta}^{\frac{p-2}{2}}g^{r}\phi^{k}\langle J\nabla
_{b}\varphi^{l},\nabla_{b}\varphi_{0}^{l}\rangle d\mu dt\nonumber \label{107}\\
&  \leq \frac{\varepsilon}{4}\int_{Q_{R}}f_{\delta}^{\frac{p-2}{2}}g^{r}%
\phi^{k}|\nabla_{b}\varphi_{0}^{l}|^{2}d\mu dt+\frac{16}{\varepsilon}%
\int_{Q_{R}}f_{\delta}^{\frac{p-2}{2}}g^{r}\phi^{k}f_{\delta}d\mu dt.
\end{align}
From (\ref{101})-(\ref{107}), we have
\begin{align}
&  \frac{1}{1+r}{\sup}_{t\leq t_{0}}\int_{B(x_{0},R)}g^{1+r}\phi^{k}d\mu
+\frac{r}{2}\int_{Q_{R}}f_{\delta}^{\frac{p-2}{2}}g^{r-1}|\nabla_{b}g|^{2}%
\phi^{k}d\mu dt\nonumber \label{109}\\
&  +(p-2)(\frac{3r}{4}-1)\int_{Q_{R}}f_{\delta}^{\frac{p-4}{2}}g^{r-1}%
\langle \nabla_{b}f_{\delta},\nabla_{b}\varphi^{l}\rangle^{2}\phi^{k}d\mu
dt\nonumber \\
&  -(p-2)r^{2}\varepsilon \int_{Q_{R}}f_{\delta}^{\frac{p-4}{2}}g^{r-2}%
e_{0}\langle \nabla_{b}\varphi^{l},\nabla_{b}f_{\delta}\rangle^{2}\phi^{k}d\mu
dt\nonumber \\
&  +\frac{p-2}{2}\int_{Q_{R}}f_{\delta}^{\frac{p-4}{2}}g^{r}\phi^{k}%
|\nabla_{b}f_{\delta}|^{2}d\mu dt+2\int_{Q_{R}}f_{\delta}^{\frac{p-2}{2}%
}|\nabla_{b}^{2}\varphi^{l}|^{2}g^{r}\phi^{k}d\mu dt\nonumber \\
&  +\frac{7\varepsilon}{4}\int_{Q_{R}}f_{\delta}^{\frac{p-2}{2}}|\nabla
_{b}\varphi_{0}^{l}|^{2}g^{r}\phi^{k}d\mu dt-(p-2)r^{2}\varepsilon^{3}%
\int_{Q_{R}}f_{\delta}^{\frac{p-2}{2}}g^{r-2}e_{0}^{2}|\nabla_{b}\varphi^{l}
_{0}|^{2}\phi^{k}d\mu dt\nonumber \\
&  -4(p-2)r\varepsilon^{2}\int_{Q_{R}}f_{\delta}^{\frac{p-2}{2}}|\nabla
_{b}\varphi^{l}_{0}|^{2}e_{0}g^{r-1}\phi^{k}d\mu dt\nonumber \\
&  \leq \frac{k}{1+r}\int_{Q_{R}}g^{1+r}\phi^{k-1}|\frac{\partial \phi}{\partial
t}|d\mu dt+(p-2)k^{2}\varepsilon \int_{Q_{R}}f_{\delta}^{\frac{p-2}{2}}%
g^{r}e_{0}|\nabla_{b}\phi|^{2}\phi^{k-2}d\mu dt\\
&  +\frac{2k^{2}}{r}\int_{Q_{R}}f_{\delta}^{\frac{p-2}{2}}g^{1+r}|\nabla
_{b}\phi|^{2}\phi^{k-2}d\mu dt+(p-2)\frac{k^{2}}{4}\int_{Q_{R}}f_{\delta
}^{\frac{p-2}{2}}g^{1+r}|\nabla_{b}\phi|^{2}\phi^{k-2}d\mu dt\nonumber \\
&  +\frac{16}{\varepsilon}\int_{Q_{R}}f_{\delta}^{\frac{p}{2}}g^{r}\phi
^{k}d\mu dt+C\int_{Q_{R}}[f_{\delta}^{\frac{p}{2}}+f_{\delta}^{\frac{p+2}{2}%
}+\varepsilon f_{\delta}^{\frac{p}{2}}e_{0}]g^{r}\phi^{k}d\mu dt.\nonumber
\end{align}
For any fixed $t$, then we choose $\frac{1}{r^{4}}<\varepsilon<\frac{1}{r^{3}%
}$ such that
\begin{equation}
(r^{2}-c)\varepsilon e_{0}\leq cf_{\delta} \label{110}%
\end{equation}
for some positive constant $c$ and
\[
\varepsilon r^{2}e_{0}\leq cf_{\delta}+c\varepsilon e_{0}.
\]
That is
\[
\frac{\varepsilon e_{0}}{f_{\delta}+\varepsilon e_{0}}\leq \frac{c}{r^{2}}%
\leq \frac{\frac{3r}{4}-1}{r^{2}}%
\]
for some $r$ is large enough and then
\begin{equation}
(\frac{3r}{4}-1)g-r^{2}\varepsilon e_{0}\geq0. \label{111}%
\end{equation}
From (\ref{110}), we have
\[
f_{\delta}\geq \frac{e_{0}}{r^{2}}\geq \varepsilon e_{0},
\]
so we have
\begin{equation}
f_{\delta}=\frac{1}{2}f_{\delta}+\frac{1}{2}f_{\delta}\geq \frac{1}{2}%
f_{\delta}+\frac{1}{2}\varepsilon e_{0}=\frac{1}{2}g \label{112}%
\end{equation}
and
\begin{equation}
\frac{7}{4}(f_{\delta}+\varepsilon e_{0})^{2}-(p-2)r^{2}\varepsilon^{2}%
e_{0}^{2}-4(p-2)\varepsilon e_{0}(f_{\delta}+\varepsilon e_{0})\geq0.
\label{113}%
\end{equation}
From (\ref{109})-(\ref{113}),
\begin{align}
&  {\sup}_{t\leq t_{0}}\int_{B(x_{0},R)}g^{1+r}\phi^{k}d\mu+\int_{Q_{R}%
}g^{r+\frac{p}{2}-2}|\nabla_{b}g|^{2}\phi^{k}d\mu dt\nonumber \\
&  \leq C\int_{Q_{R}}g^{1+r}\phi^{k-1}|\frac{\partial \phi}{\partial t}|d\mu
dt+C\int_{Q_{R}}g^{r+\frac{p}{2}}\phi^{k-2}|\nabla_{b}\phi|^{2}d\mu
dt\label{320}\\
&  +C\int_{Q_{R}}(g^{r+\frac{p}{2}}+g^{r+\frac{p}{2}+1})\phi^{k}d\mu
dt,\nonumber
\end{align}
where $C$ is a positive constant depending on $k,M,N,r,\varepsilon,p$. Thus
\begin{align}
&  {\sup}_{t\leq t_{0}}\int_{B(x_{0},R)}g^{1+r}\phi^{k}d\mu+\int_{Q_{R}%
}g^{r+\frac{p}{2}-2}|\nabla_{b}g|^{2}\phi^{k}d\mu dt\nonumber \\
&  \leq C\int_{Q_{R}}g^{1+r}\phi^{k-1}|\frac{\partial \phi}{\partial t}|d\mu
dt+C\int_{Q_{R}}g^{r+\frac{p}{2}}\phi^{k-2}(|\nabla_{b}\phi|^{2}+\phi^{2})d\mu
dt\label{cch10}\\
&  +C\int_{Q_{R}}g^{r+\frac{p}{2}+1}\phi^{k}d\mu dt.\nonumber
\end{align}
Let $h=\frac{r+\frac{p}{2}+1}{r+1}$ and $l=\frac{r+\frac{p}{2}+1}{\frac{p}{2}%
}$ and fix $k\geq l$ large enough. By H\"{o}lder's inequality and Young's
inequality, we have
\begin{align}
&  \int_{Q_{R}}g^{1+r}\phi^{k-1}|\frac{\partial \phi}{\partial t}|d\mu
dt\leq(\int_{Q_{R}}g^{r+\frac{p}{2}+1}\phi^{k}d\mu dt)^{\frac{r+1}{r+\frac
{p}{2}+1}}\nonumber \\
&  (\int_{Q_{R}}\phi^{k-\frac{2r+p+2}{p}}|\frac{\partial \phi}{\partial
t}|^{\frac{2r+p+2}{p}}d\mu dt)^{\frac{p}{2r+p+2}}\label{cch7}\\
&  \leq \int_{Q_{R}}g^{r+\frac{p}{2}+1}\phi^{k}d\mu dt+C\int_{Q_{R}}%
\phi^{k-\frac{r+3}{2}}|\frac{\partial \phi}{\partial t}|^{\frac{2r+p+2}{p}}d\mu
dt\nonumber \\
&  \leq \int_{Q_{R}}g^{r+\frac{p}{2}+1}\phi^{k}d\mu dt+CR(R-\rho)^{-\frac
{2r+p+2}{p}}|B(x_{0},R)|,\nonumber
\end{align}
where $C$ is a constant depending on $M$, $p$ and $r$. By choosing
$k\geq2(r+\frac{p}{2}+1)$, and using H\"{o}lder's inequality and Young's
inequality, we have
\begin{align}
&  \int_{Q_{R}}g^{\frac{p}{2}+r}[\phi^{k}+\phi^{k-2}|\nabla_{b}\phi|^{2}]d\mu
dt\label{cch8}\\
&  \leq \int_{Q_{R}}g^{r+\frac{p}{2}+1}\phi^{k}d\mu dt+CR[1+(R-\rho
)^{-(2r+p+2)}]|B(x_{0},R)|,\nonumber
\end{align}
where $C$ is a constant depending on $M,p$ and $r$. It follows from
(\ref{cch10}), (\ref{cch7}) and (\ref{cch8}) that
\begin{align}
&  \sup_{t\leq t_{0}}\int_{B(x_{0},R)}g^{1+r}\phi^{k}d\mu+\int_{Q_{R}%
}g^{r+\frac{p}{2}-2}|\nabla_{b}g|^{2}\phi^{k}d\mu dt\leq C_{1}\int_{Q_{R}%
}g^{r+\frac{p}{2}+1}\phi^{k}d\mu dt\label{cch101}\\
&  +C_{1}R[1+(R-\rho)^{-(p+2r+2)}+(R-\rho)^{-\frac{2r+p+2}{p}}]|B(x_{0}%
,R)|,\nonumber
\end{align}
where $C_{1}$ is a constant depending on $M,N,p$ and $r$. We recall the
following CR Sobolev inequality (\cite{j}, \cite[Theorem C]{lu} or
\cite[(1.1)]{dls})%
\begin{equation}
(\int_{B(x_{0},R)}\omega^{q}d\mu)^{\frac{1}{q}}\leq c(\int_{B(x_{0},R)}%
|\nabla_{b}\omega|^{p}d\mu)^{\frac{1}{p}}, \label{sobolev}%
\end{equation}
where $\omega \in C_{0}^{1}(B(x_{0},R))$ and $c$ is a constant, provided that
$1\leq p<Q=2n+2$ and $\frac{1}{p}-\frac{1}{q}=\frac{1}{Q}.$ In particular for
$q=2,$ we have
\begin{equation}
(\int_{B(x_{0},R)}\omega^{2}d\mu)\leq c(\int_{B(x_{0},R)}|\nabla_{b}%
\omega|^{\frac{2(n+1)}{n+2}}d\mu)^{\frac{n+2}{n+1}}. \label{sobo}%
\end{equation}
Applying this inequality to $\omega=g^{\frac{r+\frac{p}{2}+1}{2}}\psi
^{\frac{k}{2}}$,
\begin{align}
&  \int_{B(x_{0},R)}g^{r+\frac{p}{2}+1}\psi^{k}d\mu \nonumber \label{cch9}\\
&  \leq c(\int_{B(x_{0},R)}|\nabla_{b}(g^{\frac{r+\frac{p}{2}+1}{2}}%
\psi^{\frac{k}{2}})|^{\frac{2n+2}{n+2}}d\mu)^{\frac{n+2}{n+1}}\nonumber \\
&  =c(\int_{B(x_{0},R)}|\frac{r+\frac{p}{2}+1}{2}g^{\frac{r+\frac{p}{2}-1}{2}%
}\psi^{\frac{k}{2}}\nabla_{b}g+\frac{k}{2}g^{\frac{r+\frac{p}{2}+1}{2}}%
\psi^{\frac{k-2}{2}}\nabla_{b}\psi|^{\frac{2n+2}{n+2}}d\mu)^{\frac{n+2}{n+1}%
}\nonumber \\
&  \leq C(\int_{B(x_{0},R)}|g^{\frac{r+\frac{p}{2}-1}{2}}\psi^{\frac{k}{2}%
}\nabla_{b}g|^{\frac{2n+2}{n+2}}d\mu)^{\frac{n+2}{n+1}}+C(\int_{B(x_{0}%
,R)}|g^{\frac{r+\frac{p}{2}+1}{2}}\psi^{\frac{k-2}{2}}\nabla_{b}\psi
|^{\frac{2n+2}{n+2}}d\mu)^{\frac{n+2}{n+1}}\\
&  \leq C_{2}(\int_{B(x_{0},R)}g^{n+1}d\mu)^{\frac{1}{n+1}}(\int_{B(x_{0}%
,R)}\psi^{k}g^{r+\frac{p}{2}-2}|\nabla_{b}g|^{2}d\mu)\nonumber \\
&  +C_{2}(\int_{B(x_{0},R)}|g^{\frac{r+\frac{p}{2}+1}{2}}\psi^{\frac{k-2}{2}%
}\nabla_{b}\psi|^{\frac{2n+2}{n+2}}d\mu)^{\frac{n+2}{n+1}}\nonumber \\
&  =C_{2}||g(\cdot,t)||_{L^{n+1}(B(x_{0},R))}(\int_{B(x_{0},R)}\psi
^{k}g^{r+\frac{p}{2}-2}|\nabla_{b}g|^{2}d\mu)\nonumber \\
&  +C_{2}(\int_{B(x_{0},R)}|g^{\frac{r+\frac{p}{2}+1}{2}}\psi^{\frac{k-2}{2}%
}\nabla_{b}\psi|^{\frac{2n+2}{n+2}}d\mu)^{\frac{n+2}{n+1}},\nonumber
\end{align}
where $C_{2}$ is a constant depending on $M,p$ and $r$. Multiply (\ref{cch9})
by $\eta^{k}(t)$ and integrating on $[t_{0}-R,t_{0}]$,
\begin{align}
&  \int_{Q_{R}}g^{r+\frac{p}{2}+1}\phi^{k}d\mu dt\nonumber \label{cch102}\\
&  \leq C_{2}\sup_{t_{0}-R<t\leq t_{0}}||g(\cdot,t)||_{L^{n+1}(B(x_{0}%
,R))}(\int_{Q_{R}}\phi^{k}g^{r+\frac{p}{2}-2}|\nabla_{b}g|^{2}d\mu dt)\\
&  +C_{2}\int_{t_{0}-R}^{t_{0}}(\int_{B(x_{0},R)}g^{(r+\frac{p}{2}%
+1)\frac{n+1}{n+2}}\psi^{\frac{(k-2)(n+1)}{n+2}}\eta^{\frac{k(n+1)}{n+2}%
}|\nabla_{b}\psi|^{\frac{2n+2}{n+2}}d\mu)^{\frac{n+2}{n+1}}dt.\nonumber
\end{align}
Set $\varepsilon_{1}=\frac{1}{2C_{1}C_{2}}$. Suppose that
\[
||g(\cdot,t)||_{L^{n+1}(B(x_{0},R))}\leq \varepsilon_{1}.
\]
It follows from (\ref{cch101}) and (\ref{cch102}) that
\begin{align*}
&  \sup_{t_{0}-R\leq t\leq t_{0}}\int_{B(x_{0},R)}g^{1+r}\phi^{k}d\mu
+\int_{Q_{R}}g^{r+\frac{p}{2}+1}\phi^{k}d\mu dt\\
&  \leq C\int_{t_{0}-R}^{t_{0}}(\int_{B(x_{0},R)}g^{(r+\frac{p}{2}%
+1)\frac{n+1}{n+2}}\psi^{\frac{(k-2)(n+1)}{n+2}}\eta^{\frac{k(n+1)}{n+2}%
}|\nabla_{b}\psi|^{\frac{2n+2}{n+2}}d\mu)^{\frac{n+2}{n+1}}dt\\
&  +CR[1+(R-\rho)^{-(p+2r+2)}+(R-\rho)^{-\frac{2r+p+2}{p}}]|B(x_{0},R)|.
\end{align*}
From the definition of $\phi$, we obtain,
\begin{align}
&  \sup_{t_{0}-\rho \leq t\leq t_{0}}\int_{B(x_{0},\rho)}g^{1+r}d\mu
+\int_{Q_{\rho}}g^{r+\frac{p}{2}+1}d\mu dt\nonumber \label{cch12}\\
&  \leq C(R-\rho)^{-2}\int_{t_{0}-R}^{t_{0}}(\int_{B(x_{0},R)}g^{\frac
{(r+\frac{p}{2}+1)(n+1)}{n+2}}d\mu)^{\frac{n+2}{n+1}}dt\nonumber \\
&  +CR[1+(R-\rho)^{-\frac{2r+p+2}{p}}+(R-\rho)^{-(p+2r+2)}]|B(x_{0},R)|\\
&  \leq C[1+\int_{t_{0}-R}^{t_{0}}(\int_{B(x_{0},R)}f^{\frac{(r+\frac{p}%
{2}+1)(n+1)}{n+2}}d\mu)^{\frac{n+2}{n+1}}dt],\nonumber
\end{align}
where $C$ is a constant depending on $M,N,R,r,p$ and $\rho$. Let
$d=\frac{(r+3)(n+1)}{n+2}+\frac{1}{n+2}(r+1)$. By using H\"{o}lder
inequality,
\begin{align}
&  \int_{t_{0}-\rho}^{t_{0}}(\int_{B(x_{0},\rho)}g^{\frac{(r+\frac{p}%
{2}+1)(n+1)}{n+2}+\frac{1}{n+2}(r+1)}d\mu)^{\frac{n+2}{n+1}}%
dt\nonumber \label{cch13}\\
&  \leq \int_{t_{0}-\rho}^{t_{0}}[\int_{B(x_{0},\rho)}g^{(r+\frac{p}{2}+1)}%
d\mu(\int_{B(x_{0},\rho)}g^{r+1}d\mu)^{\frac{1}{n+1}}]dt\\
&  \leq \sup_{t_{0}-\rho \leq t\leq t_{0}}(\int_{B(x_{0},\rho)}g^{1+r}%
d\mu)^{\frac{1}{n+1}}\int_{Q_{\rho}}g^{(r+\frac{p}{2}+1)}d\mu dt.\nonumber
\end{align}
From (\ref{cch12}) and (\ref{cch13}), we obtain
\begin{align}
&  \int_{t_{0}-\rho}^{t_{0}}(\int_{B(x_{0},\rho)}g^{\frac{(r+\frac{p}%
{2}+1)(n+1)}{n+2}+\frac{1}{n+2}(r+1)}d\mu)^{\frac{n+2}{n+1}}dt\label{cch14}\\
&  \leq C[1+\int_{t_{0}-R}^{t_{0}}(\int_{B(x_{0},R)}g^{\frac{(r+\frac{p}%
{2}+1)(n+1)}{n+2}}d\mu)^{\frac{n+2}{n+1}}dt]^{\frac{n+2}{n+1}},\nonumber
\end{align}
where $C$ is a constant depending on $M,N,R,r,p$ and $\rho$. We set
$\theta=1+\frac{1}{n+1}$ and for $s\in N$, $R_{s}=(1+2^{-s})\frac{R_{0}}{2}$.
Define $r_{s}=(\frac{2n-p+4}{2})\theta^{s}-1$, and $a_{s}=(r_{s}+\frac{p}%
{2}+1)\frac{n+1}{n+2}$. From the definitions of $r_{s},a_{s}$, we have
$a_{s+1}=(r_{s}+\frac{p}{2}+1)\frac{n+1}{n+2}+(r_{s}+1)\frac{1}{n+2}$. If
$\rho=R_{s+1}$, $R=R_{s}$ and $r=r_{s}$ in (\ref{cch14}), we have
\begin{align}
&  \int_{t_{0}-R_{s+1}}^{t_{0}}(\int_{B(x_{0},R_{s+1})}g^{a_{s+1}}d\mu
)^{\frac{n+2}{n+1}}dt\label{cch15}\\
&  \leq C_{s}[1+\int_{t_{0}-R_{s}}^{t_{0}}(\int_{B(x_{0},R_{s})}g^{a_{s}}%
d\mu)^{\frac{n+2}{n+1}}dt]^{\frac{n+2}{n+1}},\nonumber
\end{align}
where $C_{s}$ is a constant depending on $M,N$ and $s$. Since $a_{s}%
\rightarrow \infty$ when $s\rightarrow \infty$, by interating (\ref{cch15}), we
have for any $q\geq1$
\begin{equation}
||g||_{L^{q}([t_{0}-\frac{R_{0}}{2},t_{0}]\times B(x_{0},\frac{R_{0}}{2}%
))}\leq C, \label{cch16}%
\end{equation}
where the constant $C$ depends on $M,N,q$ and $||g(\cdot,t)||_{L^{a_{0}}(M)}$.
Since $a_{0}=n+1$ and by hypothesis $\sup_{0\leq t\leq T^{\prime}}%
||g(\cdot,t)||_{L^{n+1}(M)}\leq \varepsilon_{1}$, then $C$ depends only on
$M,N$ and $q$. Following the same steps as in proof of (\ref{cch16}) where we
take $\psi^{k}(x)$ instead of $\phi^{k}(x,t)$ and we integrate on $[0,t]$ (for
$t\in \lbrack0,1]$) instead of $[t_{0}-R,t_{0}]$, we can obtain the following
inequality:
\begin{equation}
||g||_{L^{q}([0,1]\times B(x_{0},\frac{R_{0}}{2}))}\leq C, \label{cch17}%
\end{equation}
where $C$ depends on $M,N,q$ and $K$ which is any positive constant such that
$||\nabla u_{0}||_{L^{\infty}(M)}\leq K$. \ Since $M$ is compact, we know that
(\ref{33a}) is true by using (\ref{cch16}) and (\ref{cch17}). To prove
(\ref{33b}), we proceed as in (\ref{cch17}) where we integrate on $[0,t]$ with
$t\in \lbrack0,T^{\prime})$.

$(ii)\ $Note that as in $(i),$ we assume
\[
\sup_{0\leq t<T^{\prime}}||g(t,\cdot)||_{L^{n+1}(M)}\leq \varepsilon_{1}%
\]
in order to get the control of the term $\int_{Q_{R}}g^{r+\frac{p}{2}+1}%
\phi^{k}d\mu dt$ as in (\ref{cch102}). However, if \ the sectional curvature
of $(N,g_{ij})$ is nonpositive
\[
K^{N}\leq0,
\]
it follows from (\ref{32B}) that we do not need to estimate this term any
more. Then we have the estimates (\ref{33a}) and (\ref{33b}) without
assumption (\ref{2016}) if the sectional curvature of $(N,g_{ij})$ is nonpositive.
\end{proof}

\begin{lemma}
\label{l34} Let $(M^{2n+1},J,\theta)$ be a closed Sasakian manifold and
$(N,g)$ be a compact Riemannian manifold. Let $u_{0}\in C^{2,\alpha}(M,N)$,
$0<\alpha<1,$ $||\nabla u_{0}||_{L^{\infty}(M)}\leq K$ and $\varphi_{\delta}$
is the solution of the regularized equation (\ref{4a}).

$(i)$\ There exists $\varepsilon_{1}>0$ depending on $K,\ M,N$ such that if
\begin{equation}
\sup_{0\leq t<T^{\prime}}||g(t,\cdot)||_{L^{n+1}(M)}\leq \varepsilon_{1},
\label{34a}%
\end{equation}
then
\begin{equation}
||g||_{L^{\infty}([0,T^{\prime})\times M)}\leq C, \label{34}%
\end{equation}
where $C$ is a constant depending on $K,M$ and $N$.

$(ii)$ In addition, if \ the sectional curvature of $(N,g_{ij})$ is
nonpositive
\[
K^{N}\leq0,
\]
then (\ref{34}) holds without the smallness assumption (\ref{34a}).
\end{lemma}

\begin{proof}
$(i)\ $Let $\phi(x,t)=\psi(x)\eta(t)$ and $Q_{R}$ is as in the proof of Lemma
\ref{l33}. Multiplying inequality (\ref{32A}) by $g^{r}\phi^{2}$, $r\geq0$ and
similar to (\ref{101}), we have
\begin{align}
&  \frac{1}{1+r}{\sup}_{t\leq t_{0}}\int_{B(x_{0},R)}g^{1+r}\phi^{2}d\mu
+\frac{r}{2}\int_{Q_{R}}f_{\delta}^{\frac{p-2}{2}}g^{r-1}|\nabla_{b}g|^{2}%
\phi^{2}d\mu dt\nonumber \\
&  +(p-2)(\frac{3r}{4}-1)\int_{Q_{R}}f_{\delta}^{\frac{p-4}{2}}g^{r-1}%
\langle \nabla_{b}f_{\delta},\nabla_{b}\varphi^{l}\rangle^{2}\phi^{2}d\mu
dt\nonumber \\
&  -(p-2)r^{2}\varepsilon \int_{Q_{R}}f_{\delta}^{\frac{p-4}{2}}g^{r-2}%
e_{0}\langle \nabla_{b}\varphi^{l},\nabla_{b}f_{\delta}\rangle^{2}\phi^{2}d\mu
dt\nonumber \\
&  +\frac{p-2}{2}\int_{Q_{R}}f_{\delta}^{\frac{p-4}{2}}g^{r}\phi^{2}%
|\nabla_{b}f_{\delta}|^{2}d\mu dt+2\int_{Q_{R}}f_{\delta}^{\frac{p-2}{2}%
}|\nabla_{b}^{2}\varphi^{l}|^{2}g^{r}\phi^{2}d\mu dt\nonumber \\
&  +\frac{7\varepsilon}{4}\int_{Q_{R}}f_{\delta}^{\frac{p-2}{2}}|\nabla
_{b}\varphi_{0}^{l}|^{2}g^{r}\phi^{2}d\mu dt-(p-2)r^{2}\varepsilon^{3}%
\int_{Q_{R}}f_{\delta}^{\frac{p-2}{2}}g^{r-2}e_{0}^{2}|\nabla_{b}\varphi
_{0}^{l}|^{2}\phi^{2}d\mu dt\nonumber \\
&  -4(p-2)r\varepsilon^{2}\int_{Q_{R}}f_{\delta}^{\frac{p-2}{2}}|\nabla
_{b}\varphi_{0}^{l}|^{2}e_{0}g^{r-1}\phi^{2}d\mu dt\nonumber \\
&  \leq \frac{2}{1+r}\int_{Q_{R}}g^{1+r}\phi|\frac{\partial \phi}{\partial
t}|d\mu dt+4(p-2)\varepsilon \int_{Q_{R}}f_{\delta}^{\frac{p-2}{2}}g^{r}%
e_{0}|\nabla_{b}\phi|^{2}d\mu dt\nonumber \\
&  +\frac{8}{r}\int_{Q_{R}}f_{\delta}^{\frac{p-2}{2}}g^{1+r}|\nabla_{b}%
\phi|^{2}d\mu dt+(p-2)\int_{Q_{R}}f_{\delta}^{\frac{p-2}{2}}g^{1+r}|\nabla
_{b}\phi|^{2}d\mu dt\nonumber \\
&  +\frac{16}{\varepsilon}\int_{Q_{R}}f_{\delta}^{\frac{p}{2}}g^{r}\phi
^{2}d\mu dt+C\int_{Q_{R}}[f_{\delta}^{\frac{p}{2}}+f_{\delta}^{\frac{p+2}{2}%
}+\varepsilon f_{\delta}^{\frac{p}{2}}e_{0}]g^{r}\phi^{2}d\mu dt.\nonumber
\end{align}
For any fixed $t$, we choose $\varepsilon$ small enough so that
\begin{align}
&  \frac{1}{1+r}{\sup}_{t\leq t_{0}}\int_{B(x_{0},R)}g^{1+r}\phi^{2}d\mu
+\frac{r}{2}\int_{Q_{R}}g^{r+\frac{p}{2}-2}|\nabla_{b}g|^{2}\phi^{2}d\mu
dt\nonumber \\
&  \leq C\int_{Q_{R}}g^{1+r}\phi|\frac{\partial \phi}{\partial t}|d\mu
dt+C\int_{Q_{R}}g^{r+\frac{p}{2}}|\nabla_{b}\phi|^{2}d\mu dt+C\int_{Q_{R}%
}(g^{r+\frac{p}{2}}+g^{r+\frac{p}{2}+1})\phi^{2}d\mu dt, \label{cch11}%
\end{align}
where $C$ is a positive constant depending on $M,N,r$ and $p$. Since
\[
\int_{Q_{R}}g^{r+\frac{p}{2}-2}|\nabla_{b}g|^{2}\phi^{2}d\mu dt\geq \frac
{8}{(p+2r)^{2}}\int_{Q_{R}}|\nabla_{b}(g^{\frac{r}{2}+\frac{p}{4}}\phi
)|^{2}d\mu dt-\frac{16}{(p+2r)^{2}}\int_{Q_{R}}g^{r+\frac{p}{2}}|\nabla
_{b}\phi|^{2}d\mu dt,
\]
so
\begin{align}
&  \frac{1}{1+r}{\sup}_{t\leq t_{0}}\int_{B(x_{0},R)}g^{1+r}\phi^{2}d\mu
+\frac{4r}{(p+2r)^{2}}\int_{Q_{R}}|\nabla_{b}(g^{\frac{r}{2}+\frac{p}{4}}%
\phi)|^{2}d\mu dt\nonumber \\
&  \leq C\int_{Q_{R}}g^{1+r}\phi|\frac{\partial \phi}{\partial t}|d\mu
dt+(C+\frac{8r}{(p+2r)^{2}})\int_{Q_{R}}g^{r+\frac{p}{2}}|\nabla_{b}\phi
|^{2}d\mu dt\label{cch11}\\
&  +C\int_{Q_{R}}(g^{r+\frac{p}{2}}+g^{r+\frac{p}{2}+1})\phi^{2}d\mu
dt.\nonumber
\end{align}
From the definition of $\phi$,
\begin{align}
&  \sup_{t\leq t_{0}}\int_{B(x_{0},R)}g^{1+r}\phi^{2}d\mu+\int_{Q_{R}}%
|\nabla_{b}(g^{\frac{r+2}{2}}\phi)|^{2}\mu dt\nonumber \\
&  \leq C[(1+r)\int_{Q_{R}}g^{\frac{p}{2}+r}d\mu dt+(1+r)\int_{Q_{R}%
}g^{1+\frac{p}{2}+r}\phi^{2}d\mu dt\label{cch24}\\
&  +(R-\rho)^{-1}\int_{Q_{R}}g^{1+r}d\mu dt+(R-\rho)^{-2}\int_{Q_{R}}%
g^{\frac{p}{2}+r}d\mu dt],\nonumber
\end{align}
where $C$ is a constant depending on $M,N$. Let $q=(1+\frac{1}{n+1}%
)r+(\frac{p}{2}+\frac{1}{n+1})$. By using H\"{o}lder inequality,
\begin{equation}%
\begin{array}
[c]{l}%
\int_{Q_{\rho}}g^{q}d\mu dt\\
\leq \sup_{t\leq t_{0}}(\int_{B(x_{0},R)}g^{1+r}\phi^{2}d\mu)^{\frac{1}{n+1}%
}(\int_{t_{0}-R}^{t_{0}}(\int_{B(x_{0},R)}(g^{\frac{r}{2}+\frac{p}{4}}%
\phi)^{\frac{2(n+1)}{n}}d\mu)^{\frac{n}{n+1}}dt).
\end{array}
\label{cch22}%
\end{equation}
From the CR Sobolev inequality (\ref{sobolev}) again (with $p=2$)$,$ we have
\begin{equation}
(\int_{B(x_{0},R)}V^{\frac{2(n+1)}{n}}d\mu)^{\frac{n}{2n+2}}\leq
c(\int_{B(x_{0},R)}|\nabla_{b}V|^{2}d\mu)^{\frac{1}{2}}, \label{cch2016}%
\end{equation}
for all $V\in C_{0}^{1}(B(x_{0},R))$. Let $V=g^{\frac{r}{2}+\frac{p}{4}}\phi
$,
\begin{equation}
(\int_{B(x_{0},R)}(g^{\frac{r}{2}+\frac{p}{4}}\phi)^{\frac{2(n+1)}{n}}%
d\mu)^{\frac{n}{n+1}}\leq c\int_{B(x_{0},R)}|\nabla_{b}(g^{\frac{r}{2}%
+\frac{p}{4}}\phi)|^{2}d\mu. \label{cch23}%
\end{equation}
It follows from (\ref{cch22}) and (\ref{cch23}) that
\begin{equation}
\int_{Q_{\rho}}g^{q}d\mu dt\leq c\sup_{t\leq t_{0}}(\int_{B(x_{0},R)}%
g^{1+r}\phi^{2}d\mu)^{\frac{1}{n+1}}(\int_{Q_{R}}|\nabla_{b}(g^{\frac{r}%
{2}+\frac{p}{4}}\phi)|^{2}\mu dt). \label{cch25}%
\end{equation}
and from (\ref{cch24}) and (\ref{cch25}) that
\begin{align}
&  \int_{Q_{\rho}}g^{q}d\mu dt\leq C[[(1+r)+(R-\rho)^{-2}]\int_{Q_{R}}%
g^{\frac{p}{2}+r}d\mu dt\label{cch26}\\
&  +(R-\rho)^{-1}\int_{Q_{R}}g^{1+r}d\mu dt+(1+r)\int_{Q_{R}}f^{\frac{p}%
{2}+r+1}d\mu dt]^{\frac{n+2}{n+1}}.\nonumber
\end{align}
By using H\"{o}lder inequality and Young's inequality in (\ref{cch26}),%
\begin{equation}%
\begin{array}
[c]{l}%
\int_{Q_{\rho}}g^{q}d\mu dt\leq C\{[(1+r)^{\frac{2r+p+2}{2r+p}}+(R-\rho
)^{\frac{-4r-2p-4}{2r+p}}\\
+(R-\rho)^{\frac{-2r-p-2}{2r+2}}]\int_{Q_{R}}g^{\frac{p}{2}+r+1}d\mu
dt+|Q_{R}|\}^{\frac{n+2}{n+1}},
\end{array}
\label{cch27}%
\end{equation}
where $C$ depends on $M,N$ and $p$. Now we use a Moser iteration process to
obtain a uniform estimate. Let $R_{s}=(1+2^{-s})\frac{R_{0}}{2}$ and
$\theta=1+\frac{1}{n+1}$. Define $r_{s}$ and $a_{s}$ by
\[
r_{s}=\theta^{s}+n\quad \text{and}\quad a_{s}=r_{s}+\frac{p}{2}+1
\]
so that
\[
a_{s+1}=\theta r_{s}+\frac{p}{2}+\frac{1}{n+1}.
\]
We apply (\ref{cch27}) with $\rho=R_{s+1}$, $R=R_{s}$ and $r=r_{s}$ to find
\[
\int_{Q_{R_{s+1}}}g^{a_{s+1}}d\mu dt\leq C\{ \int_{Q_{R_{s}}}g^{a_{s}}d\mu
dt+C\}^{\frac{n+2}{n+1}},
\]
where $C$ is a positive constant depending only on $M,N$. By Moser iteration
process, we have
\begin{equation}
||g||_{L^{\infty}(Q_{\frac{R_{0}}{2}})}\leq C\{ \int_{Q_{R_{0}}}g^{a_{0}}d\mu
dt+C\}. \label{cch29}%
\end{equation}
Now we suppose that
\[
\sup_{0\leq t\leq T^{\prime}}||g(\cdot,t)||_{L^{n+1}(M)}\leq \varepsilon_{1}%
\]
with $\varepsilon_{1}=\varepsilon,$ where $\varepsilon$ is the constant in
Lemma \ref{l32} corresponding to $q=a_{0}=n+\frac{p}{2}+2$. By Lemma
\ref{l33}, the right hand side of (\ref{cch29}) is bounded by a constant
depending on $K,M,N$. Since $M$ is compact, we have
\[
||g||_{L^{\infty}([0,T^{\prime})\times M)}\leq C.
\]

$(ii)$ The same proof without the smallness assumption if the sectional
curvature of $(N,g_{ij})$ is nonpositive, we have the estimate (\ref{34}).
\end{proof}

\bigskip

Now we are ready to the proof of \textbf{Theorem \ref{T31} : }

\begin{proof}
$(i)\ $First, let us show that
\[
\sup_{0\leq t\leq T^{\prime}}||g(\cdot,t)||_{L^{n+1}(M)}\leq \varepsilon_{1},
\]
where $g=f_{\delta}+\varepsilon e_{0}(\varphi)$. To this end, we set
\[
T^{\ast}=\sup \{ \tilde{T}\in \lbrack0,T^{\prime}):\sup_{0\leq t\leq \tilde{T}%
}||g(\cdot,t)||_{L^{n+1}(M)}\leq \varepsilon_{1}\}.
\]
We want to prove that $T^{\ast}=T^{\prime}$. \ If $T^{\ast}<T^{\prime}$, we
have
\[
\sup_{0\leq t\leq T^{\ast}}||g(\cdot,t)||_{L^{n+1}(M)}\leq \varepsilon_{1}%
\]
So from Lemma \ref{l34}, we have
\[
||g(\cdot,t)||_{L^{\infty}([0,T^{\ast}]\times M)}\leq C.
\]
From this inequality, we have
\[
f_{\delta}\leq C\quad \text{\textrm{and}}\quad \varepsilon e_{0}\leq C.
\]
Take $\varepsilon$ small enough, we have
\[
\varepsilon e_{0}\leq \frac{\varepsilon_{0}}{2}%
\]
where $\varepsilon_{0}$ satisfying
\[
E_{p}(u_{0})=\frac{1}{p}\int_{M}|\nabla_{b}u_{0}|^{p}d\mu \leq \varepsilon_{0}.
\]
On the other hand, we have for all $t\in \lbrack0,T^{\ast}]$,
\begin{equation}
||g(\cdot,t)||_{L^{n+1}(M)}^{n+1}\leq||g||_{L^{\infty}([0,T^{\ast}]\times
M)}^{n+1-\frac{p}{2}}||g(\cdot,t)||_{L^{\frac{p}{2}}(M)}^{\frac{p}{2}}.
\label{cch32}%
\end{equation}
Since $(x+y)^{p}\leq C_{p}(x^{p}+y^{p})$ for all $p>1$ and we can choose
$\delta<\frac{\varepsilon_{0}}{2}$, we have
\begin{equation}
||g||_{L^{\frac{p}{2}}(M)}^{\frac{p}{2}}=\int_{M}[|\nabla_{b}\varphi_{\delta
}|^{2}+\delta+\varepsilon e_{0}]^{\frac{p}{2}}dv\leq C_{p}\int_{M}[|\nabla
_{b}u_{0}|^{p}dv+\varepsilon_{0}^{\frac{p}{2}}Vol(M)\leq C_{1}\varepsilon_{0}.
\label{cch33}%
\end{equation}

From (\ref{cch32}) and (\ref{cch33}), we have
\[
\sup_{0\leq t\leq T^{\ast}}||g(\cdot,t)||_{L^{n+1}(M)}^{n+1}\leq C^{n-1}%
C_{1}\varepsilon_{0}.
\]
We can take $\varepsilon_{0}$ such that $[C^{n-1}C_{1}\varepsilon_{0}%
]^{\frac{1}{n+1}}<\frac{\varepsilon_{1}}{2}$, so we have
\[
\sup_{0\leq t\leq T^{\ast}}||g(\cdot,t)||_{L^{n+1}(M)}\leq \lbrack C^{n}%
C_{1}\varepsilon_{0}]^{\frac{1}{n+1}}\leq \frac{\varepsilon_{1}}{2}.
\]
Since $f$ is continuous, we see that there exists $h>0$ such that
\[
\sup_{0\leq t\leq T^{\ast}+h}||g(\cdot,t)||_{L^{n+1}(M)}\leq \frac
{3\varepsilon_{1}}{4}\leq \varepsilon_{1},
\]
so contradicting the definition of $T^{\ast}$. Then $T^{\ast}=T^{\prime}$ and
$\sup_{0\leq t\leq T^{\ast}+h}||g(\cdot,t)||_{L^{n+1}(M)}\leq \varepsilon_{1}$.
By Lemma \ref{l34}, we have
\[
||g(\cdot,t)||_{L^{\infty}([0,T^{\prime})\times M)}\leq C.
\]
That is
\[
||\nabla_{b}\varphi_{\delta}||_{L^{\infty}([0,T^{\prime})\times M)}\leq
C\  \  \  \mathrm{and}\  \  \ ||\mathbf{T}\varphi_{\delta}||_{L^{\infty
}([0,T^{\prime})\times M)}\leq C,\
\]
where $C$ is a constant depending on $K,M$ and $N$.

$(ii)\  \ $The same proof as in $(i)$ except we do not need the smallness
assumption (\ref{t31a}) due to $K_{N}\leq0.$

$(iii)$\ By integration by parts, we compute as in \cite[Lemma 3.1.]{cc1}, one
has
\[
\frac{d}{ds}E_{p,\delta}(\varphi_{\delta}(\cdot,s))=-\int_{M}|\partial
_{s}\varphi_{\delta}(\cdot,s)|^{2}d\mu.
\]
Integrating the above equality over $[0,t]$ gives%

\[
\int_{0}^{t}\int_{M}|\partial_{s}\varphi_{\delta}|^{2}(x,s)d\mu ds+E_{p,\delta
}(\varphi_{\delta}(\cdot,t))=E_{p,\delta}(u_{0})\leq E_{p,1}(u_{0}%
),\quad \forall \,t\in \lbrack0,T_{\delta}).
\]

\end{proof}

\bigskip

Now we are ready to prove the global existence and asymptotic convergence of
the $p$-pseudoharmonic map heat flow. We will give the proof of
\textbf{Theorem \ref{t2} }which is based on \cite{fr2} and \cite{ccw}. The
proof of \textbf{Theorem \ref{t1} }is the similar to the proof of\textbf{
Theorem \ref{t2}.}

The proof of \textbf{Theorem \ref{t2} :}

\begin{proof}
We first observe that the proof is standard once we have the estimate
(\ref{t31}) plus CR divergence theorem and Green's identity as in \cite{ccw}.
\ We will sketch the proof and refer to the last section of \cite{fr2} for
more details.

$(i)\ $Global Existence : Firstly, it follows from \cite{hs}, \cite{d},
\cite{df} and \cite{ch} that there exist constants $\beta \in(0,1)\ $and $C$
depending only on $M,\ N,\ p$ such that%
\begin{equation}
||\varphi_{\delta}||_{C^{\beta}(M\times \lbrack0,T_{\delta}),N)}+||\nabla
_{b}\varphi_{\delta}||_{C^{\beta}(M\times \lbrack0,T_{\delta}),N)}\leq C.
\label{41}%
\end{equation}
This and the theory of parabolic equations (\cite{lsu}) imply
\[
\sup_{0\leq t<T_{\delta}}(||\varphi_{\delta}(\cdot,t)||_{C^{2+\alpha}%
(M,N)}+||\partial_{t}\varphi_{\delta}(\cdot,t)||_{C^{\alpha}(M,N)}\leq
C_{\delta}%
\]
for $0<\alpha<1$. It is clear that the maximal time will be infinite
\[
T_{\delta}=\infty.
\]
Moreover, by the energy inequality (\ref{t31b}), there exist a sequence
$\delta_{k}\rightarrow0$ and $\varphi \in C^{\beta}(M\times \lbrack0,\infty),N)$
and $\nabla_{b}\varphi \in C^{\beta}(M\times \lbrack0,\infty),N)$ such that
\[
\left \{
\begin{array}
[c]{l}%
(i)\ ||\nabla_{b}\varphi||_{C^{\beta}(M\times \lbrack0,\infty),N)}\leq C_{1},\\
(ii)\  \partial_{t}\varphi \in L^{2}(M\times \lbrack0,\infty)),
\end{array}
\right.
\]
and
\[
\left \{
\begin{array}
[c]{ccl}%
\varphi_{\delta_{k}} & \rightarrow & \varphi \  \  \  \mathrm{in\ }C_{loc}%
^{\beta^{\prime}}(M\times \lbrack0,\infty),N)\text{\  \textrm{for\ all} }%
\beta^{\prime}<\beta,\\
\nabla_{b}\varphi_{\delta_{k}} & \rightarrow & \nabla_{b}\varphi
\  \  \  \  \mathrm{in\ }C_{loc}^{\beta^{\prime}}(M\times \lbrack0,\infty
),N)\text{\  \textrm{for\ all} }\beta^{\prime}<\beta,\\
\partial_{t}\varphi_{\delta_{k}} & \rightarrow & \partial_{t}\varphi
\  \  \  \text{\textrm{weakly in} }L^{2}(M\times \lbrack0,\infty)).
\end{array}
\right.
\]
Now by multiplying $\psi \in C_{0}^{\infty}(M\times \lbrack0,\infty))$ for
(\ref{4a}) and integrating on $M\times \lbrack0,\infty),$ we have%
\begin{equation}%
\begin{array}
[c]{l}%
\int_{0}^{\infty}\int_{M}\partial_{t}\varphi \psi d\mu dt+\int_{0}^{\infty}%
\int_{M}|\nabla_{b}\varphi|^{p-2}\nabla_{b}\varphi \nabla_{b}\psi d\mu dt\\
=2\int_{0}^{\infty}\int_{M}|\nabla_{b}\varphi|^{p-2}h^{\alpha \bar{\beta}%
}\tilde{\Gamma}_{ij}^{k}\varphi_{\alpha}^{i}\varphi_{\bar{\beta}}^{j}\psi d\mu
dt,
\end{array}
\label{42}%
\end{equation}
as $\delta_{k}\rightarrow0.$ hence $\varphi$ is a weak solution of (\ref{4}).
The uniqueness is also standard. We refer to the last section of \cite{fr2}
for more details.

$(ii)\ $Asymptotic Convergence : Since $\partial_{t}\varphi \in L^{2}%
(M\times \lbrack0,\infty)),$ it follows from (\ref{42}) that for almost $t>0$
and $\phi \in C^{\infty}(M)$%
\[
\int_{M}\partial_{t}\varphi \phi d\mu+\int_{M}|\nabla_{b}\varphi|^{p-2}%
\nabla_{b}\varphi \nabla_{b}\phi d\mu=2\int_{M}|\nabla_{b}\varphi
|^{p-2}h^{\alpha \bar{\beta}}\tilde{\Gamma}_{ij}^{k}\varphi_{\alpha}^{i}%
\varphi_{\bar{\beta}}^{j}\phi d\mu
\]
and then there exists a sequence $t_{k}\rightarrow \infty$ such that
\[
||\partial_{t}\varphi(t_{k},\cdot)||_{L^{2}(M)}\rightarrow0.
\]
Furthermore, we have
\[
||\varphi(t_{k},\cdot)||_{C^{1+\beta}(M,N)}\leq C
\]
and then
\[
\varphi(t_{k},\cdot)\rightarrow \varphi_{\infty}(\cdot)\text{ }\  \mathrm{in\ }%
C^{1+\beta^{\prime}}(M,N)\text{\  \textrm{for\ all} }\beta^{\prime}<\beta.
\]
Moreover $\varphi_{\infty}$ is a weakly $p$-harmonic map and $\varphi_{\infty
}\in C^{1+\beta}(M,N)$ with%
\[
E_{p}(\varphi_{\infty})\leq E_{p}(u_{0}).
\]

In addition if we choose $\overline{\varepsilon_{0}}>0$ such that $E_{p}%
(u_{0})\leq \overline{\varepsilon_{0}}$, then for a fixed $q>Q=2n+2$%
\[
\int_{M}|\nabla_{b}\varphi_{\infty}|^{q}d\mu \leq pC^{q-p}\overline
{\varepsilon_{0}}%
\]
and for$\ N\subset \mathbf{R}^{l}$
\[
||\varphi_{\infty}||_{S^{1,q}(M,\mathbf{R}^{l})}\leq C\overline{\varepsilon
_{0}}.
\]
Here we apply Poincare inequality by changing the origin in $\mathbf{R}^{l}$
such that $\int_{M}\varphi_{\infty}d\mu=0.$ It follows from CR Sobolev
embedding theorem (\cite[Theorem 21.1.]{fs}, \cite{jl}) that
\[
||\varphi_{\infty}||_{C^{\frac{1}{2}(1-\frac{Q}{q})}(M,\mathbf{R}^{l})}\leq
C\overline{\varepsilon_{0}}.
\]
Thus if we choose $\overline{\varepsilon_{0}}$ small enough, $\varphi_{\infty
}(M)$ is contained in a convex geodesic ball $U$ of $N.$ By following the
first name author's previous result as in \cite[Lemma 4.2.]{cc2} (or
\cite{gor}), we consider the composite function $G:=F\circ \varphi$ \ with a
smooth function $F\ $defined on $U$. \ In particular, let us choose
$F(y)=\sum_{j}(y_{j})^{2}$, where $\{y_{j}\}_{j}$ is the coordinate system of
$U$. \ Thus
\[%
\begin{array}
[c]{l}%
\operatorname{div}_{b}(|\nabla_{b}\varphi_{\infty}|^{p-2}\nabla_{b}G)\\
=\sum_{k}F_{k}[\operatorname{div}_{b}(|\nabla_{b}\varphi|^{p-2}\nabla
_{b}\varphi^{k})+2|\nabla_{b}\varphi|^{p-2}h^{\alpha \bar{\beta}}\tilde{\Gamma
}_{ij}^{k}\varphi_{\alpha}^{i}\varphi_{\bar{\beta}}^{j}]\\
\  \  \ +|\nabla_{b}\varphi_{\infty}|^{p-2}\sum_{i,j}\sum_{\alpha}F_{ij}%
(\varphi_{\infty}^{i})_{\alpha}(\varphi_{\infty}^{j})_{\overline{\alpha}}\\
=\sum_{k}F_{k}[\operatorname{div}_{b}(|\nabla_{b}\varphi|^{p-2}\nabla
_{b}\varphi^{k})+2|\nabla_{b}\varphi|^{p-2}h^{\alpha \bar{\beta}}\tilde{\Gamma
}_{ij}^{k}\varphi_{\alpha}^{i}\varphi_{\bar{\beta}}^{j}]+|\nabla_{b}%
\varphi_{\infty}|^{p}.
\end{array}
\]
Since $\varphi_{\infty}$ is a weakly $p$-harmonic map, we obtain by
integrating both sides
\[
\int_{M}|\nabla_{b}\varphi_{\infty}|^{p}d\mu=0
\]
and we \ deduce that $\nabla_{b}\varphi_{\infty}=0.$ This completes the proof
of Theorem \ref{t2}.
\end{proof}

\end{document}